\documentclass[11pt]{article}

\input{headers.tex}

\begin{document}

\title{On the consistency of Fr\'echet means in deformable models for curve and image analysis}

\author{J\'er\'emie Bigot and Benjamin Charlier \vspace{0.2cm} \\
Institut de Math\'ematiques de Toulouse\\
Universit\'e de Toulouse et CNRS (UMR 5219)  \vspace{0.2cm} \\
{\small {\tt Jeremie.Bigot@math.univ-toulouse.fr, Benjamin.Charlier@math.univ-toulouse.fr} }}
\maketitle

\thispagestyle{empty}

\begin{abstract}
A new class of statistical deformable models is introduced to study high-dimensional curves or images.  In addition to the standard measurement error term, these deformable models include an extra error term modeling  the individual variations in intensity around a  mean pattern. It is shown that an appropriate tool for statistical inference in such models is the notion of sample Fr\'echet means, which leads to estimators of the deformation parameters and the mean pattern. The main contribution of this paper is to study how the behavior of these estimators depends on the number $n$ of design points and the number $J$ of observed curves (or images). Numerical experiments are given to illustrate the finite sample performances of the procedure.
\end{abstract}



\bigskip

\noindent{\bf Acknowledgements -}   We would like to thank Dominique Bakry for fruitful discussions. Both authors would like to thank the Center for Mathematical Modeling and the CNRS for financial support and excellent hospitality while visiting Santiago where part of this work was carried out. We are very much indebted to the referees and the Associate Editor for their constructive comments and remarks that resulted in a major revision of the original manuscript.

\section{Introduction}\label{part:intro}

\subsection{A statistical deformable model for curve and image analysis}

In many applications, one observes a set of curves or grayscale images which are high-dimensional data. In such settings, it is reasonable to assume that the data at hand $Y_{j}^{\ell}$, denoting the $\ell$-th observation for the $j$-th curve (or image), satisfy the following regression model:
\begin{equation} \label{eq:model}
Y_{j}^{\ell} = f_{j}(t_{\ell}) +\sigma \varepsilon_{j}^{\ell}, \quad j=1,\ldots,J, \; \text{ and }\; \ell = 1,\ldots,n,
\end{equation}
where $f_{j} : \Omega \longrightarrow \R$ are unknown regression functions (possibly random) with $\Omega$ a convex subset of $\R^{d}$, the $t_{\ell}$'s are non-random points in $\Omega$ (deterministic design), the error terms $\varepsilon_{j}^{\ell}$ are i.i.d.\ normal variables with zero mean and variance $1$, and $\sigma > 0$. In this paper, we will suppose that the $f_{j}$'s are random elements which vary around the same mean pattern. Our goal is to estimate such a mean pattern and to study the consistency of the proposed estimators in various asymptotic settings:  either when both the number $n$ of design points  and the number $J$ of curves (or images)  tend to infinity, or when $n$ (resp.\ $J$) remains fixed while $J$ (resp.\ $n$)  tends to infinity. 

In many situations, data sets of curves or images exhibit a source of geometric variations in time or shape. In such settings, the usual Euclidean mean $\bar{Y}^{\ell} = \frac{1}{J} \sum_{j=1}^{J} Y_{j}^{\ell} $  in model \eqref{eq:model} cannot be used to recover a meaningful mean pattern. Indeed, consider the following simple model of randomly shifted curves (with $d=1$)  which is commonly used in many applied areas such as neuroscience  \cite{TIR} or biology  \cite{MR1841413},
\begin{equation} \label{eq:modelshift}
f_{j}(t_{\ell}) = f(t_{\ell} - \btheta_{j}^{\ast}),  \quad j=1,\ldots,J, \; \text{ and }\;\ell = 1,\ldots,n,
\end{equation}
where $f : \Omega \longrightarrow \R$ is the mean pattern of the observed curves, and the $\btheta_{j}^{\ast}$'s are i.i.d.\ random variables in $\RR$ with density $g$ and independent of the $ \varepsilon_{j}^{\ell}$'s. In model (\ref{eq:modelshift}), the shifts $\btheta_{j}^{\ast}$  represent a source of variability in time. However, in (\ref{eq:modelshift}) the  Euclidean mean is not a consistent estimator of the mean pattern $f$ since by the law of large numbers
$$
\lim_{J\to \infty} \bar{Y}^{\ell}  =  \lim_{J\to \infty} \frac{1}{J} \sum_{j=1}^J f(t_\ell - \btheta_j^{\ast}) = \int f(t_{\ell} - \btheta)g(\btheta)d\btheta \quad a.s.
$$

The randomly shifted curves model \eqref{eq:modelshift} is close to the perturbation model introduced by \cite{MR1108330} in shape analysis for the study of consistent estimation of a mean pattern from a set of random planar shapes. The mean pattern  to estimate in  \cite{MR1108330} is called a population mean, but to stress the fact that it comes from a perturbation model \cite{HuckemannSJS} uses the term perturbation mean. To achieve consistency in such models, a Procrustean procedure is used in \cite{MR1108330}, which  leads to the statistical analysis of sample Fr\'echet means  \cite{fre} which are extensions of the usual Euclidean mean to non-linear spaces using non-Euclidean metrics. For random variables belonging to a nonlinear manifold, a well-known example is the computation of  the mean of a set of planar shapes in the Kendall's shape space \cite{kendall} which leads to the  Procrustean means studied in  \cite{MR1108330}.   Consistent estimation of a mean planar shape has been studied by various authors, see e.g.\ \cite{MR1108330,MR1436569,MR1891212,huilling98,lekum}. A detailed study of some properties of the Fr\'echet mean in finite dimensional Riemannian manifolds (such as consistency and uniqueness) has been performed in \cite{MR0501230,MR1370296,batach1,batach2,HuckemannSJS,HuckemannAOS,Afsari} .

The main goal of this paper is to introduce statistical deformable models for curve and image analysis that are analogue to Goodall's perturbation models  \cite{MR1108330}, and to build consistent estimators of a mean pattern in such models. Our approach is inspired by Grenander's pattern theory which considers that the curves or images $f_{j}$ in model \eqref{eq:model} are obtained through the deformation of a  mean pattern  by a Lie group action  \cite{Gre,gremil}. In the last decade, there has been a growing interest in transformation Lie groups  to model the geometric variability of  images, and the study of the properties of such deformation groups  is now an active field of research (see e.g.\   \cite{miltrou,trouyou} and references therein). 
There is also currently a growing interest in statistics on the use of Lie group actions to analyze geometric modes of variability of a data set \cite{Munk1,Munk2}.

To describe more formally geometric variability, denote by $L^{2}(\Omega)$ the set of square integrable real-valued functions on $\Omega$, and by $\PP$ an open subset of $\RR^{p}$. To the set $\PP$, we associate a parametric family of operators $(T_{\btheta})_{\btheta \in \PP}$  such that for each $\btheta \in \PP$ the operator $T_{\btheta}: L^{2}(\Omega) \longrightarrow L^{2}(\Omega)$ represents a geometric deformation (parametrized by $\btheta$) of a curve or an image. Examples of such deformation operators include the cases of:
\begin{description}
\item[-]  {\em Shifted curves:} $T_{\btheta} f(t) := f(t-\btheta),$ with $\Omega = [0,1]$, $f \in L_{per}^{2}([0,1])$ (the space of periodic functions in $L^{2}([0,1])$ with period 1) and $\PP$ an open set of $\RR$.

\item[-] {\em Rigid deformation of two-dimensional images:} 
$$T_{\btheta} f(t) :=  f \left(e^a  R_{\alpha} t - b \right), \quad   \mbox{ for } \btheta = (a,\alpha,b) \in \PP,$$
with  $\Omega = \R^{2}$, $\PP \subset \R \times \R \times \R^2$ 
where  $R_{\alpha} = \left(\begin{array}{cc} \cos(\alpha) &  -\sin(\alpha) \\   \sin(\alpha) & \cos(\alpha) \end{array}\right)$ is a rotation matrix in $\R^{2}$, $e^a$ is an isotropic scaling and $b$ a translation in $\RR^2$.

\item[-] {\em Deformation by a Lie group action:} the two above cases are examples of a Lie group action on the space $L^{2}(\Omega)$ (see \cite{MR1834454} for an introduction to Lie groups). More generally, assume that $G$ is a connected Lie group of dimension $p$ acting on $\Omega$, meaning that for any $(g,t) \in G \times \Omega$ the  action $\cdot$ of $G$ onto $\Omega$ is such that $g \cdot t \in \Omega$. In general, $G$ is not a linear space but can be locally parametrized by a its Lie algebra $\mathcal{G} \simeq \R^p$ using the exponential map  $\exp : \mathcal{G} \to G$.  If $\PP \subset\R^p$. This leads  
for $(\btheta,f) \in \PP \times L^{2}(\Omega)$ to define the deformation  operators
$$
T_{\btheta} f(t) := f \left( \exp(\btheta) \cdot t \right).
$$

\item[-] {\em Non-rigid deformation of curves or images:} assume that one can construct a family  $(\psi_{\btheta})_{\btheta \in \PP}$ of parametric diffeomorphisms of $\Omega$ (see e.g.\ \cite{BGL09}). Then, for $(\btheta,f) \in \PP \times L^{2}(\Omega)$, define the deformation operators
$$
T_{\btheta} f(t) := f \left( \psi_{\btheta}(t) \right).
$$
\end{description}
Then, in model \eqref{eq:model}, we assume that the $f_{j}$'s have a certain homogeneity in structure in the sense that there exists some $f \in L^2(\Omega)$ such that 
\begin{equation}\label{eq:modT}
f_{j}(t) = T_{\btheta_{j}^{\ast}}\big[f + Z_{j}\big](t), \quad \mbox{ for all } t \in \Omega,  \; \text{ and }\; j=1,\ldots,J,
\end{equation}
where $\btheta^{\ast}_{j} \in \PP, \ j=1,\ldots,J$ are i.i.d.\ random variables (independent of the $\varepsilon_{j}^{\ell}$'s) with an unknown density $g$ with compact support $\Theta$ included in $\PP$ satisfying:
\begin{hyp}\label{hyp:Theta}
The density $g$ of  the $\btheta^{\ast}_{j}$'s is continuously differentiable on $\PP$ and has a compact support $\Theta$ included in $\PP\subset\R^{p}$. We assume that $\Theta$  can be written
\begin{equation}
\Theta = \left\{ \btheta = (\theta^{1},\ldots,\theta^{p})\in\R^p,\ |\theta^{p_1}| \leq \rho,\  1 \leq p_1 \leq p \right\}
\end{equation}
where $\rho > 0$.
\end{hyp}

 The  function  $f$ in model \eqref{eq:modT} represents the unknown mean pattern of the $f_{j}$'s. The $Z_j$'s are supposed to be independent of the $\varepsilon_{j}^{\ell}$'s and are i.i.d.\ realizations of a second order centered Gaussian process $Z$ taking its values in $ L^2(\Omega)$.  The $Z_{j}$'s represent the individual variations in intensity around $f$, while the random operators  $T_{\btheta_{j}}$ model geometric deformations in time or space.  Then, if we assume that the $T_{\btheta}$'s are linear operators, equation (\ref{eq:modT}) leads to the following {\em statistical deformable model} for curve or image analysis
\begin{equation} \label{eq:modeldeform}
Y_{j}^{\ell} = T_{\btheta_{j}^{\ast}} f(t_{\ell}) + T_{\btheta_{j}^{\ast}} Z_{j}(t_{\ell}) +\sigma \varepsilon_{j}^{\ell}, \quad j=1,\ldots,J,  \; \text{ and }\; \ell = 1,\ldots,n,
\end{equation}
where $\varepsilon_{j}^{\ell}$ are i.i.d.\ normal variables with zero mean and variance $1$.

Model \eqref{eq:modeldeform} could be also called a {\it perturbation model} using the terminology in \cite{MR1108330,HuckemannSJS} for shape analysis. To be more precise, let $\bY \in \R^{n \times 2}$ be a set of $n$ points in $\RR^2$ representing a planar shape. Define a deformation operator $T_{\btheta}$ for  $\btheta = ( a, \alpha ,b ) \in \Theta =\R\times [0,2\pi] \times  \R^2$ acting on $\R^{n\times 2}$ in the following way
$$
T_{\btheta} \bY = e^{a} \bY R_{\alpha} + \1_n b' , \mbox{ where } R_{\alpha} = \left(\begin{array}{cc}\cos(\alpha) & -\sin(\alpha)\\ \sin(\alpha) & \cos(\alpha)\end{array} \right),
$$
and $\1_n = (1,\ldots,1)' \in \RR^{n}$. Consistent estimation of a mean shape has been first studied in \cite{MR1108330} when a set of random shapes $\bY_{1},\ldots,\bY_{J}$ is drawn from the following perturbation model 
\begin{equation}\label{eq:modelKendall}
\bY_j = T_{\btheta_{j}^*} (\mu + \bzeta_j), \; j=1,\ldots,J.
\end{equation}

Model \eqref{eq:modelKendall} is similar to the statistical deformable model (\ref{eq:modeldeform}), where $\mu \in \R^{n \times 2}$ is the unknown perturbation mean to estimate, and $\bzeta_j$ are i.i.d.\ random vectors in $\R^{n \times 2}$ with zero mean. Nevertheless, there exists major differences between our approach and the one in  \cite{MR1108330}. First, in model (\ref{eq:modeldeform}), the deformations parameters $\btheta_{j}^{\ast}$  are assumed to be random variables following an unknown distribution, whereas they are just nuisance parameters in model \eqref{eq:modelKendall} for shape analysis, see \cite{MR1108330,MR1436569}. In some applications (e.g.\ in biomedical imaging  \cite{Joshi04unbiaseddiffeomorphic}), it is of interest to reconstruct the unobserved parameters $\btheta_{j}^{\ast}$ and to estimate their distribution. One of the main contribution of this paper is  then to  construct upper and lower bounds for the  estimation of such deformation parameters. Moreover, in model (\ref{eq:modeldeform}), they are too additive error terms, whereas the model \eqref{eq:modelKendall} only include the error term $\bzeta_j$. In model (\ref{eq:modeldeform}), the $\varepsilon_{j}^{\ell}$ is an additive noise modeling the errors in the measurements, while the $Z_j$'s model (possibly smooth) variations in intensity of the individuals around the mean pattern $f$. 

In \cite{
MR1436569}, the authors studied the relationship between isotropicity of  the additive noise $\bzeta_j$ and the convergence of Procrustean procedures to the perturbation mean $\mu$ as $J \to + \infty$. It is shown in \cite{
MR1436569} that, for isotropic errors,  Procrustean means are consistent, but that, for non-isotropic errors, they may not converge to $\mu$. For a recent discussion on the issues of consistency of sample Procrustes means in perturbation models and extension to non-metrical Fr\'echet means, we refer to \cite{HuckemannSJS} and \cite{HuckemannAOS}. In this paper, we carefully analyze the role of the dimension $n$ and the number of samples $J$ on the consistency of Procrustean means in  model (\ref{eq:modeldeform}). To obtain consistent procedures, we  show that it is not required to impose very restrictive conditions on the error terms $Z_j$ such as  isotropicity for the  $\bzeta_j$ in \eqref{eq:modelKendall} for shape analysis. Here, the key quantity  is the dimension $n$ of the data (number of design points) which plays the central role to guarantee the converge of our estimators. This point is another major difference with the approach of statistical shape analysis \cite{MR1108330} that does not take into account the dimensionality of the shape space to analyze the consistency of Procrustean estimators.

Note that a subclass of the deformable model \eqref{eq:modeldeform} is the so-called shape invariant model (SIM) 
\begin{equation} \label{eq:SIM}
Y_{j}^{\ell} = T_{\btheta_{j}^{\ast}} f(t_{\ell}) +\sigma \varepsilon_{j}^{\ell}, \quad j=1,\ldots,J, \;\text{ and }\; \ell = 1,\ldots,n,
\end{equation}
i.e.\ without incorporating in \eqref{eq:modeldeform} the additive terms $Z_{j}$.

The goal of this paper is twofold. First, we propose a general methodology for estimating $f$ and the $\btheta^{*}_{j}$'s based on observations coming from model (\ref{eq:modeldeform}). For this purpose, we show that an appropriate tool  is the notion of sample Fr\'echet mean of a data set \cite{fre,MR0501230,batach1} that has been widely studied in shape analysis \cite{MR1108330,MR1436569,huilling98,lekum,HuckemannSJS} and more recently in biomedical imaging \cite{Joshi04unbiaseddiffeomorphic,Pennec:2006:ISR:1166859.1166868}. Secondly, we study the consistency of the resulting estimators in various asymptotic settings: either when $n$ and $J$ both tend to infinity, or when $n$ is fixed and $J\to+\infty$, or when $J$ is fixed and $n\to+\infty$.   

\subsection{Organization of the paper}

Section \ref{part:proc} contains a description of our estimating procedure and a review of previous work in mean pattern estimation. In Section \ref{part:lower}, we derive a lower bound for the quadratic risk of estimators of the deformation parameters. In Section \ref{sec:identgen}, we discuss some identifiability issues in model \eqref{eq:modeldeform}. In Section \ref{part:uppershift} we derive  consistency results for the Fr\'echet mean in the case (\ref{eq:modelshift}) of randomly shifted curves. In Section \ref{part:ab} and Section \ref{sec:gen}, we give general conditions to extend these results  to the more general deformable model (\ref{eq:modeldeform}). Section \ref{sec:simus} contains some numerical experiments. A small conclusion with some perspectives are given  in Section \ref{sec:conc}. All proofs are postponed to a technical Appendix.

\section{The estimating procedure} \label{part:proc}

%

\subsection{A dissimilarity measure based on deformation operators}

To define a notion of sample Fr\'echet mean  for curves or images, let us suppose that the family of deformation operators $(T_{\btheta})_{\btheta \in \PP}$ is invertible in the sense that there exists a family of operators $(\tilde{T}_{\btheta})_{\btheta \in \PP}$ such that for any $(\btheta,f) \in \PP \times L^{2}(\Omega) $
$$
\tilde{T}_{\btheta} f \in L^{2}(\Omega) \quad \mbox{ and } \quad  \tilde{T}_{\btheta} T_{\btheta} f = f.
$$
Then, for two functions $f,h \in L^{2}(\Omega)$ introduce the following dissimilarity measure
\begin{equation*}
d^{2}_{T}(h,f) = \inf_{\btheta \in \PP} \int_{\Omega} \left( \tilde{T}_{\btheta} h(t) -f(t) \right)^2 dt.
\end{equation*}
If $d^{2}_{T}(h,f)= 0$ then there exists $\btheta \in \PP$ such that $f =  \tilde{T}_{\btheta} h$ meaning that the functions $f$ and $h$ are equal up to a geometric deformation. Note that $d_{T}$ is not necessarily a distance on $L^{2}(\Omega)$, but it can be used to define a notion of sample Fr\'echet mean of data from model (\ref{eq:modeldeform}). For this purpose  let $\F$ denote a subspace of $L^2(\Omega)$ and suppose that $\hat f_{j}$ are smooth functions in $\F\subset L^{2}(\Omega)$ obtained from the data $Y_{j}^{\ell}$, $\ell=1,\ldots,n$  for $j=1,\ldots,J$, see Section \ref{sec:smoothshifts} and Section \ref{part:smoothing} for precise definitions. Following the definition of a Fr\'echet mean in general metric space \cite{fre}, define an estimator of the mean pattern $f$ as
\begin{equation} \label{eq:meanpattern}  
\hat{f} = \argmin_{f \in\F} \frac{1}{J} \sum_{j=1}^{J} d^{2}_{T}(\hat f_{j},f).
\end{equation}
Note that $\hat{f}$ falls into the category of non-metrical  sample Fr\'echet means whose definitions and asymptotic properties are discussed in \cite{HuckemannSJS} for random variables belonging to Riemannian manifolds. However, unlike the usual approach in shape analysis, the Fr\'echet mean \eqref{eq:meanpattern} is based on smoothed data. In what follows, we show that smoothing is a key preliminary step to obtain the convergence of $\hat{f}$ to the mean pattern $f$  in the deformable model \eqref{eq:modeldeform}. It can be easily shown that the computation of $\hat{f} $ can be done in two steps: first minimize the following criterion 
\begin{equation} \label{eq:crittheta}
(\hat{\btheta}_{1},\ldots,\hat{\btheta}_{J}) = \argmin_{ (\btheta_{1},\ldots,\btheta_{J}) \in  \Theta^J} M(\btheta_{1},\ldots,\btheta_{J}),
\end{equation}
where
\begin{equation} \label{eq:Mcrittheta}
M(\btheta_{1},\ldots,\btheta_{J}) = \frac{1}{J} \sum_{j=1}^{J}   \int_{\Omega} \bigg( \tilde{T}_{\btheta_{j}} \hat f_{j} (t) - \frac{1}{J} \sum_{j'=1}^{J} \tilde{T}_{\btheta_{j'}} \hat f_{j'}(t)  \bigg)^{2} dt,
\end{equation}
which gives an estimation of the deformation parameters $\btheta_{1}^{\ast},\ldots, \btheta_{J}^{\ast}$, and then in a second step take
\begin{equation} \label{eq:frechetmean}
\hat{f}(t) =  \frac{1}{J} \sum_{j=1}^{J} \tilde{T}_{\hat{\btheta}_{j}} \hat f_{j}(t), \quad \mbox{ for }  t \in \Omega,
\end{equation}
as an estimator of the mean pattern $f$. 

Note that this two steps procedure  belongs to the category of Procrustean methods (see e.g \cite{MR1646114,MR1108330}). A similar  approach to (\ref{eq:crittheta}) has been developed by  \cite{Joshi04unbiaseddiffeomorphic} in the context of biomedical images using diffeomorphic deformation operators.

\subsection{Previous work in mean pattern estimation and geometric variability analysis}

Estimating the mean pattern of a set of curves that differ by a time transformation is usually referred to as the curves registration problem, see e.g.\ \cite{kneipgas,B06,ramli,wanggas,liumuller}. However, in these papers, studying consistent estimators of the  mean pattern $f$ as the number of curves $J$ and design points $n$ tend to infinity is not considered. For the SIM \eqref{eq:SIM},  a semiparametric point of view has been proposed in \cite{mazaloubgam} and \cite{vimond} to  estimate non-random deformation parameters (such as shifts and amplitudes) as the number $n$ of observations per curve grows, but with a fixed number $J$ of curves. A generalisation of this semiparametric approach for two-dimensional images is proposed in \cite{BGV09}. The case of image deformations by a Lie group action is also investigated in \cite{BLV10} from a semiparametric point of view using a SIM.

In the simplest case of randomly shifted curves in a SIM, \cite{BG10} have studied minimax estimation of the mean pattern $f$ by letting only the number $J$ of curves going to infinity. Self-modelling regression (SEMOR) methods proposed by \cite{kg} are semiparametric models where each observed curve is a parametric transformation of the same regression function. However, the SEMOR approach does not incorporate a random fluctuations in intensity of the individuals  around a mean pattern $f$ through an unknown process $Z_{j}$ as in model (\ref{eq:modeldeform}). The authors in \cite{kg}  studied the consistency of  the SEMOR approach using a Procrustean algorithm. Recently, there has also been a growing interest on the development of statistical deformable models for image analysis and the construction of consistent estimators of a mean pattern, see \cite{glasbey,BGV09,BGL09,allason,kuhn-allason}.
   
 \section{Lower bounds for the estimation of the deformation parameters}\label{part:lower}

In this section, we derive non-asymptotic lower bounds for the quadratic risk of an arbitrary estimator of the deformation parameters under the following smoothness assumption  of the mapping $(\btheta, t) \longmapsto T_{\btheta}f(t)$.
 \begin{hyp}\label{hyp:op3} 
For all $\btheta = (\theta^{1},\ldots,\theta^{p}) \in\PP$, $ T_{\btheta}: L^2(\Omega)\longrightarrow  L^2(\Omega)$ is a linear operator  
 such that the function $t \longmapsto \partial_{\theta^{p_{1}}} T_{\btheta}f(t)$ exists and belongs to $L^{2}(\Omega)$ for any  $p_{1}=1,\ldots,p$. Moreover, there exists a constant $C(\Theta,f)>0$ such that $$\norm{\partial_{\theta^{p_{1}}} T_{\btheta}f }_{L^2}^2 \leq  C(\Theta,f),$$ for all $ p_1=1,\ldots,p$ and $\btheta\in\Theta$.
\end{hyp}

\subsection{Shape Invariant Model}

\begin{theo}\label{theo:VTSIM}
Consider the SIM \eqref{eq:SIM} and suppose that Assumption \ref{hyp:op3} holds. Assume that $g$ satisfies Assumption \ref{hyp:Theta}, and that $\int_{\Theta} \left\|\partial_{\btheta } \log\left(g(\btheta )\right)\right\|^2 g(\btheta )d\btheta < + \infty$. Let  $\hat{\btheta} \in \PP^{J}$ be any estimator (a measurable function of the data) of $\btheta^* = (\btheta^*_1,\ldots,\btheta^*_J)$. Then, for any $n \geq 1$ and $J \geq 1$,
\begin{align}\label{eq:VT}
\E\left[\frac{1}{J} \snorm{\hat\btheta - \btheta^*}^2_{\R^{pJ}} \right]\geq \frac{\sigma^{2} n^{-1}}{C(\Theta,f) + \sigma^{2}n^{-1}\int_{\Theta}\norm{\partial_{\btheta} \log\left(g(\btheta)\right)}^2 g(\btheta)d\btheta}.
\end{align}
where $C(\Theta,f)$ is the constant defined in Assumption \ref{hyp:op3}, and  $\norm{\cdot}_{\R^{pJ}}$ is the standard Euclidean norm in $\R^{pJ}$.
\end{theo}
\noindent 
The lower bound given in inequality \eqref{eq:VT} does not decrease as $J$ increases. Thus, if the number $n$  of design points is fixed, increasing the number $J$ of curves (or images) does not improve the quality of the estimation of the deformation parameters for any estimator $\hat\btheta$. Nevertheless, this lower bound is going to 0 as the dimension $n\to+\infty$. 

\subsection{General model} 

The main difference between  the general model \eqref{eq:modeldeform}  and the SIM \eqref{eq:SIM} is the extra error terms $T_{\btheta_j^*}Z_j$, $j=1,\ldots,J$.  In what follows, $\E_{\btheta}[\ \cdot\ ]$ denotes expectation conditionally to $\btheta\in\Theta^J$. Since the random processes $Z_j$'s are observed through the action of the random deformation operators $T_{\btheta_{j}^*}$ it is necessary to specify how the $T_{\btheta_{j}^*}$'s modify the law of the process $Z_{j}$.

\begin{hyp}\label{hyp:Zmin}
There exists  a positive semi-definite  symmetric $n \times n$ matrix $\mathbf \Sigma_n(\Theta)$  such that  the covariance matrix of $\bZ=[Z(t_\ell)]_{\ell=1}^n$ satisfies $\Ec_{\btheta} \big[\T_{\btheta}\bZ (\T_{\btheta} \bZ)'\big] = \mathbf \Sigma_n(\Theta)$. 
\end{hyp}
This assumption means that the law of the random process $Z$ is somewhat invariant by the deformation operators $T_{\btheta}$. Such an hypothesis is similar to the condition given in \cite{MR1436569} to ensure consistency of Fr\'echet mean estimators in Kendall's shape space using model similar to \eqref{eq:modeldeform} with $\sigma=0$. After a normalization step, the deformations considered in \cite{MR1436569} are rotations of the plane, and the authors in \cite{MR1436569} study the case where the law of the error term $Z$ is isotropic, that is to say, invariant by the action of rotations.
\begin{theo}\label{theo:VTGEN}
Consider the general model \eqref{eq:modeldeform}. Suppose that Assumption \ref{hyp:op3} and \ref{hyp:Zmin} hold. 
Assume that the density $g$ satisfies Assumption \ref{hyp:Theta} and that $\int_{\Theta} \left\|\partial_{\btheta } \log\left(g(\btheta )\right)\right\|^2 g(\btheta )d\btheta < + \infty$. Let  $\hat{\btheta} \in \PP^{J}$ be any estimator (a measurable function of the data) of $\btheta^*  = (\btheta^*_1,\ldots,\btheta^*_J)$. Then, for any $n \geq 1$ and $J \geq 1$, we have 
\begin{align}\label{eq:VTGEN}
\E\left[\frac{1}{J} \snorm{\hat\btheta - \btheta^*}^2_{\R^{pJ}} \right]\geq \frac{(\sigma^2 + s^2_n(\Theta)) n^{-1}}{C(\Theta,f) + (\sigma^2+ s^2_n(\Theta))n^{-1}\int_{\Theta}\left\|\partial_{\btheta_1} \log\left(g(\btheta)\right)\right\|^2 g(\btheta)d\btheta}.
\end{align}
where $C(\Theta,f)$ is the constant defined in Assumption \ref{hyp:op3}, and $s_n^2(\Theta)$ denotes the smallest eigenvalue of $\mathbf \Sigma_n(\Theta) $.
\end{theo}
 Again, the lower bound \eqref{eq:VTGEN} does not depends on $J$. Thus, increasing the number $J$  of observations does not decrease the quadratic risk of any estimator of the deformations parameters. Moreover, the lower bound \eqref{eq:VTGEN} tends to zero as $n \to + \infty$ only if $\lim_{n \to + \infty} n^{-1}s^2_n(\Theta) = 0$. 

\subsection{Application to the shifted curves model}\label{part:lowershift}

Consider the shifted curves model \eqref{eq:modelshift} with an equi-spaced design, namely
\begin{equation}
Y_{j}^\ell = f\left(\tfrac{\ell}{n} - \btheta^*_j\right) + Z_j\left(\tfrac{\ell}{n} - \btheta^*_j\right) +\sigma\varepsilon_j^\ell, \quad j=1,\ldots,J,\; \text{ and } \; \ell=1,\ldots,n. \label{eq:modrandshift}
\end{equation}
\begin{theo}\label{theo:vantreeshift}
Consider the model \eqref{eq:modrandshift}. Assume that $f$ is continuously differentiable on $[0,1]$ and that $Z$ is a centered stationary process with value in $L^2_{per}([0,1])$. Suppose that  $\Theta = [-\rho,\rho]$ with $\rho<\tfrac{1}{2}$ 
 and 
$\int_{\Theta}\left(\partial_{\btheta } \log\left(g(\btheta )\right)\right)^2 g(\btheta )d\btheta < + \infty$. Let $\hat{\btheta} \in \R^J$ be any estimator of the true random shifts $\btheta^*  = (\btheta^*_1,\ldots,\btheta^*_J)$, i.e.\ a measurable function of the data in model (\ref{eq:modrandshift}). Then, for any $n \geq 1$ and $J \geq 1$
\begin{align}\label{eq:vantreeshift}
\E\left[\frac{1}{J} \|\hat\btheta - \btheta^*\|^2_{\R^{J}} \right] \geq\frac{n^{-1}\sigma^2}{ \norm{\partial_{t} f}_{\infty}^2  + n^{-1}\sigma^2\int_{\Theta}\left(\partial_{\btheta } \log\left(g(\btheta )\right)\right)^2 g(\btheta )d\btheta },
\end{align}
where $\norm{\partial_{t} f}_{\infty} = \sup_{t \in [0,1]} \left\{ |\partial_{t} f(t) | \right\}$ 
with $\partial_{t} f$ denoting the first derivative of $ f$.
\end{theo}

\section{Identifiability conditions} \label{sec:identgen}

\subsection{The shifted curves model}\label{part:identshift}

Without any further assumptions, the randomly shifted curves model (\ref{eq:modrandshift}) is not identifiable. Indeed, if $\btheta_{0} \in \Theta$ satisfies $\btheta^{*}_{j}+\btheta_{0} \in \Theta$, $j=1,\ldots,J$, then replacing $f(\cdot)$ by $f(\cdot- \btheta_{0})$ and $\btheta^{*}_{j}$ by  $\btheta^{*}_{j}+\btheta_{0}$ does not change the formulation of  model (\ref{eq:modrandshift}). Choosing identifiability conditions amounts to impose constraints on the minimization of the criterion 
\begin{align}\label{eq:Dshift}
D(\btheta) & =  \frac{1}{J} \sum_{j=1}^{J}   \int_{\Omega} \bigg( f(t-\btheta^{*}_{j}+\btheta_{j}) - \frac{1}{J} \sum_{j'=1}^{J} f(t-\btheta^{*}_{j'}+\btheta_{j'})\bigg)^{2} dt,
\end{align}
for $\btheta = (\btheta_{1},\ldots,\btheta_{J}) \in  \Theta^{J}$, which can be interpreted as a version without noise of the criterion \eqref{eq:crittheta} using the ideal smoothers $\hat{f}_{j}(\cdot) = f(\cdot-\btheta^{*}_{j})$.  Obviously,  the criterion $D(\btheta)$ has a minimum at $\btheta^{\ast} = (\btheta_{1}^{\ast},\ldots,\btheta_{J}^{\ast})$ such that $D(\btheta^{\ast}) = 0$, but this minimizer of $D$ on $ \Theta^{J}$  is clearly not unique. If the true shifts are supposed to have zero mean (i.e.\  $\int_{\Theta} \btheta  g(\btheta)  d \btheta = 0$) it is natural to introduce the constrained set
\begin{equation}
\bTheta_0 = \{(\btheta_1,\ldots,\btheta_J)\in \Theta^J, \ \btheta_1+\ldots+\btheta_J = 0  \}.  \label{eq:Theta0}
\end{equation}
It is shown in \cite{BG10} Lemma 6, that if  $f \in L^{2}([0,1])$ is such that $\int_0^1 f(t) e^{-i2\pi t}dt \neq 0$ and if $\rho < 1/4$ (recall that $\Theta = [-\rho,\rho]$), then 
the criterion $D(\btheta)$ has a unique minimum on $\bTheta_0$ in the sense that $D(\btheta) > D(\btheta^*_{\bTheta_0})$ for all $\btheta \in \bTheta_0$ with $\btheta \neq \btheta^*_{\bTheta_0}$ where
\begin{equation}
\btheta^*_{\bTheta_0} = (\btheta^*_1 -\bar \btheta^*,\ldots,\btheta_{J}^* - \bar\btheta^*) \mbox{ with } \bar\btheta^* = \frac{1}{J}\sum_{j=1}^J \btheta_j^*. \label{eq:theta0}
\end{equation}
Under such assumptions, we will compute estimators of the random shifts by minimizing the criterion  \eqref{eq:crittheta} over the constrained set $\bTheta_0$ and not directly on $ \Theta^{J}$. Consistency of such constrained estimators will then be studied under the following identifiability conditions:
\begin{hyp}\label{hyp:f}
The mean pattern $f$ is such that $\int_0^1 f(t) e^{-i2\pi t}dt \neq 0$.
\end{hyp}
\begin{hyp}\label{hyp:g}
The support of the density $g$ is included in $[-\rho',\rho']$ for some $0 < \rho' \leq \frac{ \rho }{2} <1/4$ and is such that $\int_{\Theta} \btheta  g(\btheta)  d \btheta = 0$.
\end{hyp}
Under such assumptions, $D(\btheta)$ can be bounded from below by the quadratic function $ \frac{1}{J} \snorm{\btheta - \btheta_{\bTheta_0}^* }^{2}$ which will be an important property to derive consistent estimators.
\begin{proposition} \label{prop:D}
 Suppose that Assumptions \ref{hyp:f} and \ref{hyp:g} hold with $\rho < 1/16$. Then,  for any $\btheta = (\btheta_{1},\ldots,\btheta_{J}) \in \bTheta_0$, one has that
$$
D(\btheta) - D(\btheta^*_{\bTheta_0}) \geq C(f,\rho) \frac{1}{J} \snorm{\btheta - \btheta_{\bTheta_0}^* }^{2},
$$
where $C(f,\rho) > 0$ is a constant depending only on $f$ and $\rho$.
\end{proposition}
Assumption \ref{hyp:g} and the condition that $\rho < 1/16$ in Proposition \ref{prop:D} mean that the support of the density $g$ of the shifts is sufficiently small, and that the shifted curves $f_{j}(t) = f(t-\btheta_{j}^{\ast})$ are in some sense concentrated around the mean pattern $f$. Such an assumption of concentration of the data around the same mean pattern has been used in various papers to prove the uniqueness and the consistency of Fr\'echet means for random variables lying in a Riemannian manifold, see   \cite{MR0442975,huilling98,batach1,Afsari, WKendall}.

\subsection{The general case}\label{part:ind}

  In the case of general deformation operators, define for $\btheta = (\btheta_{1},\ldots,\btheta_{J}) \in  \Theta^{J}$ the criterion
\begin{align}\label{eq:D}
D(\btheta) =  \frac{1}{J} \sum_{j=1}^{J}   \int_{\Omega} \bigg(\tilde T_{\btheta_{j}} T_{\btheta^{*}_{j}}f(t) - \frac{1}{J} \sum_{j'=1}^{J} \tilde T_{\btheta_{j'}} T_{\btheta^{*}_{j'}}f(t)\bigg)^{2} dt.
\end{align}
Obviously, using that for all $\btheta\in\Theta$,  $\tilde T_{\btheta} T_{\btheta} f = f$, the criterion $D(\btheta)$ has a minimum at $\btheta^{\ast} = (\btheta_{1}^{\ast},\ldots,\btheta_{J}^{\ast})$ such that $D(\btheta^{\ast}) = 0$. However, without any further restrictions the minimizer of $D(\btheta)$ is not necessarily unique on $ \Theta^{J}$.  
\begin{hyp}\label{hyp:unic}
Let $\bTheta\subset\Theta^J$ such that there exists a unique $\btheta^*_{\bTheta}\in\bTheta$ satisfying  $D(\btheta^*_{\bTheta})=0$.
\end{hyp}
\noindent Then, $\bTheta$ is the set onto which we will carry the minimization of the criterion $M(\btheta)$ \eqref{eq:Mcrittheta}. In the case of shifted curves and under Assumption \ref{hyp:f} and \ref{hyp:g}, the only  set onto which  the criterion $D$ vanishes is the line $\left\{\btheta^* + \btheta_{0} \1_J ,  \; \btheta_{0} \in \R  \right\} \subset \R^J,$  where $\1_J=(1,\ldots,1)' \in\R^J$. An easy way to choose the set $\bTheta$ is to take a linear subset of $\Theta^J$, see Figure \ref{fig:identif} for an illustration. By considering the subset 
$$
\bTheta_0 = \Theta^J \cap \1^{\perp}_J = \{(\btheta_1,\ldots,\btheta_J)\in \Theta^J, \ \btheta_1+\ldots+\btheta_J = 0  \},
$$
 where ${\1_J}^\perp$ is the orthogonal of $\1_J$ in $\R^J$, then Assumption \ref{hyp:unic} is satisfied with $\btheta_{\bTheta_0}^*$ given in \eqref{eq:theta0}. More generally, if the deformation parameters $\btheta_{j}$, $j=1,\ldots,J$ are supposed to be random variables with zero mean, then optimizing $D(\btheta)$ on $\bTheta_0$ is a natural choice.  Another identifiability condition for shifted curves is proposed in \cite{mazaloubgam} and \cite{vimond} by taking 
\begin{equation} \label{eq:Theta1}
\bTheta_1 =  \Theta^J \cap {e_1}^\perp =\{ (\btheta_1,\ldots,\btheta_J)\in \Theta^J, \ \btheta_1 = 0  \}.
\end{equation}
where $e_1 = (1,0,\ldots,0)\in\R^J$. In this case, $\btheta_{\bTheta_1}^* = (0, \btheta_2^* -\btheta_1^*, \ldots, \btheta_J^* -\btheta_1^*) $. Choosing to minimize $D(\btheta)$ on $\bTheta_1$ amounts to choose the first curve as a reference onto which all the others curves are aligned, meaning that the first shift $\btheta^{\ast}_{1}$ is not random, see Figure \ref{fig:identif}.
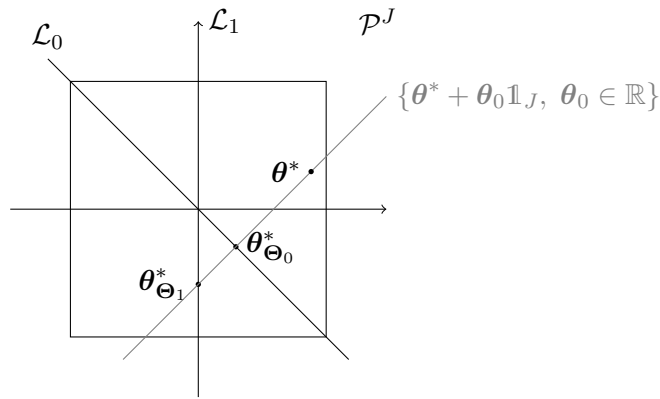
\begin{figure}[!h]
\begin{center}
\begin{tikzpicture}[scale=0.5]
\draw (-3.4,-3.4) rectangle (3.4,3.4);
\draw[->] (-5,0) -- (5,0);
\draw[->] (0,-5) -- (0,5)node[right]{$\LL_1$};
\draw (-4,4) node[above]{$\LL_0$}-- (4,-4);
\path (4,5) node[right]{$\PP^J$};
\fill (1,-1) circle (2pt) node[right]{$\btheta_{\bTheta_0}^*$};
\fill[color=black] (0,-2) circle (2pt) node[left]{$\btheta_{\bTheta_1}^*$};
\draw[color=gray] (-2,-4) -- (5,3)node[right]{$\left\{\btheta^* + \btheta_{0} \1_J ,  \; \btheta_{0} \in \R  \right\}$};
\fill (3,1) circle (2pt) node[left]{$\btheta^*$};
\end{tikzpicture}
\end{center}
\caption{Choice of identifiability conditions for shifted curves in the case $J=2$.  } \label{fig:identif}
\end{figure}

Following the classical guidelines in M-estimation (see e.g. \cite{VdW}), a necessary condition to ensure the convergence of M -estimators such as \eqref{eq:crittheta} is that the local minima of $D(\btheta)$ over $\bTheta$ are well separated from the global minimum of $D(\btheta)$ at $\btheta = \btheta_{\bTheta}^*$ (satisfying $D(\btheta_{\bTheta}^* ) = 0$). The following assumption can be interpreted in this sense.
\begin{hyp}\label{hyp:ident}
For all $\btheta\in\bTheta$ we have
\begin{equation}\label{eq:ident}
D(\btheta) - D(\btheta^*_{\bTheta}) \geq  C(\bTheta,\F) \frac{1}{J}\snorm{\btheta - \btheta^*_{\bTheta} }^2
\end{equation}
for a constant $C(\bTheta,\F) >0$  independent of $J$.
\end{hyp}
\noindent 
In the shifted curve model, Assumption \ref{hyp:ident} is verified if Assumption \ref{hyp:f} and \ref{hyp:g} hold (see Proposition \ref{prop:D}).

\section{Consistent estimation in the shifted curves model}\label{part:uppershift}

In this section, we give conditions to ensure consistency of the estimators defined in Section \ref{part:proc} in  the shifted curves model  \eqref{eq:modrandshift} with an equi-spaced design.

\subsection{The random perturbations $Z_j$}\label{part:Zshift}

Following the assummtions of Theorem \ref{theo:vantreeshift}, $Z$  will be  supposed to be a stationary process $Z$ with covariance function $R:[0,1]\longrightarrow\R$. The law of $Z$ is thus invariant by the action of a shift. Conditionally  to $\btheta_j^{\ast} \in  \Theta$, the covariance of the vector $\T_{\btheta_j^{\ast}}\bZ_j = \big[Z_j(\frac{\ell}{n} - \btheta_j^{\ast})\big]_{\ell=1}^n$ is a Toeplitz matrix equals to
\begin{equation}
 \mathbf{\Sigma}_n = \E_{\btheta_j^{\ast}}\big[\T_{\btheta_j^{\ast}}\bZ_j (\T_{\btheta_j^{\ast}} \bZ_j)'\big]= \left[ \Ec \left[Z\left(\tfrac{\ell}{n}\right) Z\left(\tfrac{\ell'}{n}\right)\right]\right]_{\ell, \ell' =1}^n =  \left[  R\left(\tfrac{|\ell-\ell'|}{n}\right) \right]_{\ell, \ell' =1}^n. \label{eq:Sigman}
\end{equation}
Let $\gamma_{\max}(\mathbf{\Sigma}_n)$ be the largest eigenvalue of the matrix $\mathbf{\Sigma}_n$. It follows from standard results on Toeplitz matrices (see e.g.\ \cite{MR1084815}) that 
\begin{equation} \label{eq:gamma}
\gamma_{\max}\big( \mathbf\Sigma_n \big) \leq \lim_{n \to + \infty} \frac{1}{n} \sum_{k=1}^{n}  \left| R\left(\tfrac{k}{n}\right) \right| =\gamma
\end{equation}
where $\gamma = \int_0^1 \abs{R(t)}dt$ is a positive constant independent of $n$ representing an upper bound of the variance of $Z$.

\subsection{Choice of the smoothed estimators $\hat{f}_{j}$} \label{sec:smoothshifts}

A convenient choice for the smoothing of the observed curves in \eqref{eq:modrandshift} is to do low-pass Fourier filtering. Let
$
{\hat{c}_{j,k}}=  \frac{1}{n} \sum_{\ell=1}^{n} Y_{j}^{\ell} e^{-i2\pi  k \frac{\ell}{n}} \mbox{ for }  k = -(n-1)/2, \ldots,(n-1)/2
$
 (assuming for simplicity that $n$ is odd), and define for a spectral cut-off parameter $\lambda\in \N$ and $t \in [0,1]$ the linear estimators
\begin{equation}\label{eq:fshift}
\hat f_j^\lambda (t) = \sum_{\abs{k}\leq\lambda}  {\hat c_{j,k} }e^{i2\pi  k t} .
\end{equation}
Then, define the Sobolev ball $H_s(A)$ of radius $A>0$ and regularity $s > 0$  as
\begin{equation} \label{eq:Hs}
H_s(A) = \Big\{ f\in L_{per}^{2}([0,1]) , \sum_{k\in\Z} (1+\abs{k}^2)^{s} \abs{c_k(f)}^2 < A \Big\}.
\end{equation}
with  $c_k(f) = \int_0^1 f(t)e^{-i2\pi kt}dt$, $k \in \Z$ for a function $f \in L_{per}^{2}([0,1])$, and take   $\F = H_s(A)$ as the smoothness class to which the mean pattern $f$  is supposed to belong.

\subsection{Consistent estimation of the random shifts}

Using low-pass filtering, and following the discussion in Section \ref{part:identshift} on identifiability issues,  the estimators of the random shifts $\btheta_{1}^{\ast},\ldots, \btheta_{J}^{\ast}$ are given by
\begin{equation} \label{eq:critthetashift1}
\hat{\btheta}^{\lambda} = (\hat{\btheta}_{1}^{\lambda},\ldots,\hat\btheta_{J}^{\lambda}) = \argmin_{ (\btheta_{1},\ldots,\btheta_{J}) \in  \bTheta_0} M_{\lambda}(\btheta_{1},\ldots,\btheta_{J}).
\end{equation}
where the criterion $M_{\lambda}(\btheta) = M_{\lambda}(\btheta_{1},\ldots,\btheta_{J})$ for $\btheta \in \Theta^{J}$ is
\begin{align*}
M_{\lambda}(\btheta) & =   \frac{1}{J} \sum_{j=1}^{J}   \int_{\Omega} \bigg( \hat f_j^\lambda (t+\btheta_{j})   - \frac{1}{J} \sum_{j'=1}^{J} \hat f_{j'}^\lambda (t+\btheta_{j'}) \bigg)^{2} dt 
\end{align*}
and $\bTheta_0$ is the constrained set defined in \eqref{eq:Theta0}.

\begin{theo}\label{theo:shift1}
Consider the model \eqref{eq:modrandshift} and let $\hat{\btheta}^{\lambda}$ be the estimator defined by \eqref{eq:critthetashift1}. Assume that $\F = H_s(A)$ for some $A > 0$ and $s \geq 1$, and that $Z$ is a centered stationary process with value in $L^2_{per}([0,1])$ and covariance function $R:[0,1]\to \R$. Suppose that Assumptions \ref{hyp:f} and  \ref{hyp:g} hold with $\rho < 1/16$. Then, for any $\lambda \geq 1$ and $x > 0$ 
$$
\P\bigg(\frac{1}{J} \lVert \hat\btheta^{\lambda} - \btheta^* \rVert^2_{\R^{J}} \geq C_{1}(\Theta,\F,f)  A_{1}(x,J,n,\lambda,\sigma^2,\gamma) + A_{2}(x,J)\bigg) \leq 4e^{-x},
$$
with
$
A_{1}(x,J,n,\lambda,\sigma^2,\gamma) =    (\sigma^2 + \gamma)\Big( \sqrt{\upsilon(x,J,n,\lambda)} + \upsilon(x,J,n,\lambda)\Big) + \Big(\sqrt{B(\lambda,n)} + B(\lambda,n) \Big)
$
and
$
A_{2}(x,J) =  \left( \sqrt{\frac{2 x}{J}} + \frac{x}{3J} \right)^2,
$
where $C_{1}(\Theta,\F,f)>0$ is constant depending only on $\Theta,\F,f$,
$
\upsilon(x,J,n,\lambda) = \frac{2 \lambda +1}{n} \left(1+ 4\frac{x}{J} + \sqrt{4\frac{x}{J}}\right)
$
,
$
B(\lambda,n) =  \frac{2 \lambda +1}{n} + \lambda^{-2s}.
$
and  $\gamma = \int_0^1 \abs{R(t)}dt$.
\end{theo}
First, remark that for fixed values of  $n$ and $\lambda$,  then $\lim_{J\to +\infty}A_{2}(x,J) = 0$. The term $A_{1}(x,J,n,\lambda,\sigma^2,\gamma)$ depends on the spectral cutoff $\lambda$ via the bias $B(\lambda,n)$ and the variance $\upsilon(x,J,n,\lambda)$ of the estimators $\hat{f}_{j}$. By choosing a sequence $\lambda = \lambda_n$ such that  $\lim_{n \to +\infty}  \lambda_n  = + \infty$ and $\lim_{n \to +\infty} \frac{\lambda_n}{n} =0$ (tradeoff between low variance and low bias)  it follows that for  fixed $ J$ and $x > 0$, then $
\lim_{n \to +\infty} A_{1}(x,J,n,\lambda_n,\sigma^2,\gamma) =0.
$
However, if $n$ remains fixed,  then
$
\lim_{J \to +\infty}A_{1}(x,J,n,\lambda,\sigma^2,\gamma) > 0.
$

Thus, Theorem \ref{theo:shift1} is consistent with the conclusions of Theorem \ref{theo:vantreeshift}, that is,  if  $n$ is fixed, then  it is not possible to estimate $\btheta^*$ by letting only $J$ grows to infinity. Hence, under the assumptions of Theorem \ref{theo:shift1}, one can only prove the convergence in probability  of $\hat\btheta^{\lambda}$ to the true shifts $\btheta^*$ by taking the double asymptotic $n \to + \infty$ and $J \to + \infty$, provided the smoothing parameter $\lambda = \lambda_n$ is well chosen.

\subsection{Consistent estimation of the mean pattern}

In the case of randomly shifted curves, the Fr\'echet mean estimator \eqref{eq:meanpattern} of $f$ is 
$
 \hat f^{\lambda} (t) = \frac{1}{J} \sum_{j=1}^J \hat f^{\lambda}_j (t + \hat\btheta^{\lambda}_{j})$.

\begin{theo}\label{theo:shift2}
Under the assumptions of Theorem \ref{theo:shift1}, for any $\lambda \geq 1$ and $x > 0$ 
$$
\P\bigg(\lVert \hat f^{\lambda} - f \rVert^2_{L^2} \geq C_{2}(\Theta,\F,f)  A_{1}(x,J,n,\lambda,\sigma^2,\gamma) +C_{3}(\Theta,f) A_{2}(x,J)\bigg) \leq 4e^{-x},
$$
where $A_1 (x,J,n,\lambda,\sigma^2,\gamma)$ and $A_{2}(x,J) $ are defined in Theorem \ref{theo:shift1},  $C_{2}(\Theta,\F,f)$ and $C_{3}(\Theta,f)$ are positive constants depending only on $ \Theta,\F,f $, and $\lVert \hat f^{\lambda} - f \rVert^2_{L^2} = \int_{0}^{1} \left| f^{\lambda}(t) - f(t) \right|^2 dt $.
\end{theo}
Similar comments to those made on the consistency of the estimators of the shifts can be made.  A double asymptotic in $n$ and $J$ is needed to show that the Fr\'echet mean $\hat f^{\lambda}$ converges in probability to the true mean pattern $f$. Moreover,  if $\lambda_n$ is too large (e.g.\ such that $\lim_{n \to +\infty} \frac{\lambda_n}{n} \neq 0$, which correspond to undersmoothing), then Theorem \ref{theo:shift2} cannot be used to prove that $ \hat f^{\lambda}$ converges to $f$ in probability. This illustrates the fact that,  to achieve consistency, a sufficient amount of pre-smoothing is necessary before computing the   Fr\'echet mean \eqref{eq:meanpattern}.

\subsection{A lower bound for the Fr\'echet mean}

From the results of Theorem \ref{theo:vantreeshift}, it is expected that the  Fr\'echet mean $\hat f^{\lambda}$ does not converge to  $f$ in the setting $n$ fixed and $J\to+\infty$. To support this argument, consider the following ideal estimator
\begin{equation}
\tilde f  (t) =  \frac{1}{J} \sum_{j=1}^J f_j( t+\hat\btheta^{\lambda}_j) =  \frac{1}{J} \sum_{j=1}^J f( t-\btheta^*_j +\hat\btheta^{\lambda}_j), \quad \text{ for all $t\in [0,1]$},
\end{equation}
where $f_j(t) = f(t-\btheta^*_j),j=1,\ldots,J$. This corresponds to the case of an ideal smoothing step from the data \eqref{eq:modrandshift} that would yield $\hat{f}_{j} = f_j$ for all $j=1,\ldots,J$. Obviously, $\tilde f  (t)$ is not an estimator since it depends on the unobserved quantities $f$ and $\btheta^*_j$, but we can consider it as a benchmark to analyse the converge of the Fr\'echet mean $\hat f^{\lambda}$ to $f$.

\begin{theo}\label{th:vantreefunct}
Suppose that the assumptions of Theorem \ref{theo:vantreeshift} are satisfied with $\rho < \frac{3}{4 \pi}$. Then, for any $n \geq 1$, there exists $J_0 \in \N$ such that $J\geq J_0$ implies
\begin{equation}
\E [\snorm{\tilde f - f }_{L^2}]  \geq C(f,\rho)  \frac{n^{-1}\sigma^2}{ \norm{\partial_{t} f}_{\infty}^2  + n^{-1}\sigma^2 \int_{\Theta}\left(\partial_{\btheta } \log\left(g(\btheta )\right)\right)^2g(\btheta )d\btheta }, \label{eq:lowerboundtildef}
\end{equation}
where the constant $C(f,\rho) > 0$ depends on $f$ and $\rho$.
\end{theo}

Hence, in the setting $n$ fixed and $J\to+\infty $, even the ideal estimator $\tilde f$  does not converge to $f$ for the expected quadratic risk. This illustrates the central role played by the dimension $n$ of the data to obtain consistent estimators.

\section{Notations and main assumptions in the general case} \label{part:ab}

\subsection{Smoothness of the mean pattern and the deformation operators}\label{part:gennote}

 In this part, the notation $(\LL_{\btheta})_{\btheta \in \PP}$ is used to denote either $(T_{\btheta})_{\btheta \in \PP}$ or their inverse $(\tilde T_{\btheta})_{\btheta \in \PP}$.
\begin{hyp} \label{ass:op12} 
For all $\btheta\in\PP$, $ \LL_{\btheta}: L^{2}(\Omega) \longrightarrow L^{2}(\Omega)$ is a linear operator satisfying $ \LL_{\btheta} f \in \F$ for all $f \in \F$. There exists a constant $C(\Theta)>0$ such that for any $f \in L^{2}(\Omega)$ and $\btheta\in\Theta$
$$
\norm{\LL_{\btheta} f }^2_{L^2} \leq  C(\Theta) \norm{f}^2_{L^2}, 
$$
and a constant $C(\F, \Theta)>0$ such that for any $f \in  \F$ and $\btheta_{1},\btheta_{2} \in \Theta$,
$$
\snorm{\tilde{T}_{\btheta_{1}}   f- \tilde{T}_{\btheta_{2}}  f }^2_{L^2}\leq  C(\F,\Theta) \norm{\btheta_1-\btheta_2}^2 .
$$
\end{hyp}

Assumption \ref{ass:op12} can be interpreted as a Lipschitz condition on the mapping $(f,\btheta) \longmapsto \mathcal{L}_{\btheta}f$. 
The first inequality, that is  $\norm{\LL_{\btheta}f}^2_{L^2} \leq C(\Theta) \norm{f}^2_{L^2}$, means that the action of the operator $\LL_{\btheta}$ does not change too much the norm of $f$ when $\btheta$ varies in $\Theta$. Such an assumption on $T_{\btheta}$ and its inverse  $\tilde{T}_{\btheta}$ forces the optimization problem \eqref{eq:crittheta} to have non trivial solutions by avoiding the functional $M(\btheta)$ in \eqref{eq:Mcrittheta} being arbitrarily small. 
 It can be easily checked that  Assumption \ref{ass:op12}  is satisfied in the case \eqref{eq:modelshift} of shifted curves with 
$
\F = H_s(A)
$
and
$
s\geq 1
$
.

\subsection{The preliminary smoothing step} \label{part:smoothing}

For $j=1,\ldots,J$ the $\hat f_{j}$'s are supposed to belong to the class of linear estimators in the sense of the following definition:

\begin{definition} \label{def:lin}
Let $\Lambda$ denote either $\N$ or $\R_{+}$ (set of smoothing parameters). To every $\lambda \in \Lambda$ is associated a non-random vector valued function $S_\lambda :  \Omega \longrightarrow \R^n$ such that for all $j=1,\ldots,J$ and all $t \in \Omega$
$$
 \hat f_{j}(t) = \hat f^\lambda_{j}(t) = \langle S_\lambda(t) , \Y_{j} \rangle, 
$$
where $\langle \cdot, \cdot \rangle$ denotes the standard inner product in $\R^{n}$ and $\Y_{j} = \big(Y_{j}^{\ell}\big)_{\ell=1}^{n} \in  \R^{n}$.
\end{definition}

\begin{hyp} \label{ass:5}
For all $\lambda \in \Lambda$ and all $\ell=1,\ldots,n$, the function $t \longmapsto  S^{\ell}_\lambda(t)$ belong to $L^2(\Omega)$, where $S^{\ell}_\lambda(t)$ denotes the $\ell$-th component of the vector $S_\lambda(t)$. Moreover, for all $\lambda \in \Lambda$, $f \in \F$ and $\btheta \in \Theta$, the function $t \longmapsto \prs{S_\lambda(t),\T_{\btheta} \f}$ belongs to $\F$ where $\T_{\btheta}  \f = \big(T_{\btheta}  f(t_{\ell}) \big)_{\ell=1}^{n}$.
\end{hyp}

In the case  (\ref{eq:modelshift})  of randomly shifted curves with an equi-spaced design, then Assumption \ref{ass:5} holds with $S_\lambda(t) = \Big[\frac{1}{n} \sum_{\abs{k}\leq\lambda} e^{i2\pi  k(t -\frac{\ell}{n})}\Big]_{\ell=1}^{n}.$ Let us now specify how the bias/variance behavior of the linear estimators $\hat f^\lambda_{j}$ depends on the smoothing parameter $\lambda$. For this, consider for some function $f \in \F$ the following regression model
$$
Y^{\ell} = f(t_{\ell})+  \sigma \varepsilon^{\ell}, \; \ell=1,\ldots,n,
$$
where the $\varepsilon_{\ell}$'s are i.i.d normal variables with zero mean and variance 1.
The performances of a linear estimator $\hat f^\lambda(t) = \prs{S_\lambda(t) , \Y }$, where $ \Y = (Y_{\ell})_{\ell=1}^n$, can be evaluated in term of the expected quadratic risk $R_{\lambda}(\hat f^\lambda,f)$ defined by
$$
R_{\lambda}(\hat f^\lambda,f) := \Ec  \big\|\big( \hat f^\lambda - f\big)\big\|_{L^2}^{2}  =  \int_\Omega  \left| B_\lambda(f,t) \right| ^2 dt + \sigma^{2}  \int_\Omega  V_\lambda(t)dt,
$$
where $B_\lambda$ and $V_\lambda$ denote the usual bias and variance of $\hat f^\lambda$ given by
$
B_\lambda(f,t) =   \prs{S_\lambda(t) , \f} -  f(t)
$
and
$
V_\lambda(t) = \norm{ S_\lambda(t)}^{2}_{\R^{n}},
$
for
$ 
t \in \Omega,
$
where $ \f = \big( f(t_{\ell}) \big)_{\ell=1}^{n}$. Define also
$
V(\lambda) = \int_{\Omega} V_\lambda(t)  dt,
$
and let us make the following assumption on the asymptotic behavior of the bias/variance of $\hat f^\lambda$:
\begin{hyp}\label{ass:smoo}
There exist a constant $\kappa(\F)>0$ and a real-valued function $\lambda\longmapsto B(\lambda)$, such that for all $f \in \mathcal F$,
$$
\norm{B_\lambda(f,\cdot) }^2_{L^2} = \norm{\prs{S_\lambda(\cdot) , \mathbf{f}} -  f(\cdot)}^2_{L^2} \leq \kappa(\F)B(\lambda).
$$
Moreover there exists a sequence of smoothing parameters $(\lambda_{n})_{n \in \N} \in \Lambda^{\N}$ with $\lim_{n \to + \infty} \lambda_{n} = + \infty$ such that
$
\lim_{n \to + \infty} B(\lambda_{n}) = 0 \mbox{ and } \lim_{n \to + \infty} V(\lambda_{n}) = 0.
$
\end{hyp}
Let us illustrate Assumption \ref{ass:smoo} in the case of shifted curves with an equi-spaced design, and a smoothing step  obtained by low-pass Fourier filtering. As in Section \ref{part:uppershift}, take $\F = H_s(A)$ defined in (\ref{eq:Hs}). In this setting, $V(\lambda) = \frac{2 \lambda + 1}{n}$. It can be also checked that $\norm{B_\lambda(f,\cdot) }^2_{L^2} \leq C(A) B(\lambda)$ for some positive constant $C(A)$ depending only on $A$, 
and $B(\lambda)  =  \frac{2 \lambda + 1}{n} + \lambda^{-2s}$. Thus, Assumption \ref{ass:smoo} holds with $\lambda_{n} = n^{\frac{1}{2s+1}}$.

\subsection{Random perturbation of the mean pattern $f$ by the $Z_{j}$'s}

\begin{hyp}\label{hyp:EigZmax}
 For any $n \geq 1$, there exists a real $\gamma_{n}(\Theta)>0$ such that for any $\btheta \in \Theta$
$$
\gamma_{\max}\big(\E_{\btheta}\big[\T_{\btheta}\bZ (\T_{\btheta} \bZ)'\big]\big ) \leq \gamma_{n}(\Theta) \quad 
$$
where $\T_{\btheta} \bZ = \big(T_{\btheta} Z(t_\ell) \big)_{\ell=1}^n \in \R^{n}$, and $\gamma_{\max}(A)$ denotes the largest eigenvalue of a symmetric matrix $A$. Moreover,
\begin{equation} \label{eq:gammaV}
\lim_{n \to \infty}\gamma_{n}(\Theta) \sqrt{ V(\lambda_n) } = 0,
\end{equation}
where $V(\lambda_n) $ is the variance defined in Assumption \ref{ass:smoo}.
\end{hyp}
Intuitively, the condition (\ref{eq:gammaV}) means that the variance of the linear smoother $S_{\lambda}(\cdot)$ has to be asymptotically smaller that the maximal correlations (measured by $\gamma_{n}(\Theta)$)  between $T_{\btheta} Z(t_{\ell})$ and $T_{\btheta} Z(t_{\ell'})$ for $\ell,\ell' = 1,\ldots,n$ and all $\btheta \in \Theta$.  In the case of randomly shifted curves with an equi-spaced design, a simple condition for which Assumption \ref{hyp:EigZmax} holds is the case where $Z$ is stationary process (see the arguments in Section \ref{part:Zshift}). 

\section{Consistency in the general case} \label{sec:gen}
 
\subsection{Consistent estimation of the deformation parameters}\label{sec:genshift}

Consider for $\lambda \in \Lambda$ the following estimator of the deformation parameters
$$
\hat \btheta^{\lambda} = \argmin_{\btheta\in \bTheta}\limits M_{\lambda}(\btheta),
$$
where
\begin{align}\label{eq:critM}
M_{\lambda} (\btheta)& = \frac{1}{J}\sum_{j=1}^{J} \int_\Omega \bigg ( \tilde T_{\btheta_{j}}\prs{ S_{\lambda}(t),\Y_{j}} - \frac{1}{J} \sum_{j'=1}^{J} \tilde T_{\btheta_{j'}}\prs{S_{\lambda}(t),\Y_{j'}}\bigg)^2dt,
\end{align}
and $\bTheta $ is the constrained set introduced in Assumption \ref{hyp:unic}. The estimator $\hat \btheta^{\lambda}$ thus depends on the choice of $\bTheta$, and it will be shown that $\hat \btheta^{\lambda}$ is a consistent estimator  of the vector $\btheta^*_{\bTheta} \in \R^{pJ}$ defined in Assumption \ref{hyp:unic}. Note that depending on the problem at hand and the choice of  the constrained set $\bTheta $, it can be shown that $\btheta^*_{\bTheta}$ is close to the true deformation parameters $\btheta^*$. For example, in the case of shifted curves, if $\bTheta = \bTheta_{0}$ defined in (\ref{eq:Theta0}) and if the density $g$ of the shifts has zero mean, then $ \btheta_{\bTheta_0} = (\btheta^*_1 -\bar \btheta^*,\ldots,\btheta_{J}^* - \bar\btheta^*)$ with $  \bar\btheta^* = \frac{1}{J}\sum_{j=1}^J \btheta_j^*$ can be shown to be close to $\btheta^*$ (see Lemma \ref{lem:Bern} in the Appendix). This allows to show the consistency of $\hat \btheta^{\lambda} $ to $\btheta^*$ as formulated in Theorem \ref{theo:shift1}. Therefore, the next result only bounds the distance between $\hat \btheta^{\lambda} $ and  $\btheta^*_{\bTheta}$. 
\begin{theo}\label{th:shift}
Consider the model \eqref{eq:modeldeform} and suppose that Assumptions  \ref{hyp:Theta}, \ref{hyp:unic}, \ref{hyp:ident} and \ref{ass:op12} to \ref{hyp:EigZmax} hold with $n \geq 1$ and $J \geq 2$. Then, for any $\lambda \in \Lambda$ and $x > 0$ 
\begin{align}
\P\bigg(\frac{1}{J} \lVert \hat\btheta^{\lambda} - \btheta_{\bTheta}^* \rVert^2_{\R^{pJ}} \geq & C_{1}(\Theta,\bTheta,\F,f) \Big[ (\gamma_{n}(\Theta)+\sigma^2)\Big( \sqrt{\upsilon(x,J,\lambda)} + \upsilon(x,J,\lambda)\Big) \nonumber \\
&  + \Big(\sqrt{B(\lambda)} + B(\lambda) \Big)\Big] \bigg) \leq 2e^{-x},  \label{eq:dev1}
\end{align}
with $C_{1}(\Theta,\bTheta,\F,f)>0$, $\upsilon(x,J,\lambda) := V(\lambda) \left(1+ 4\frac{x}{J} + \sqrt{4\frac{x}{J}}\right)$. 
\end{theo}

Using Assumptions \ref{ass:smoo} and \ref{hyp:EigZmax}, it follows that $\lim_{n\to +\infty} \gamma_{n}(\Theta) \Big( \sqrt{\upsilon(x,J,\lambda_{n})} + \upsilon(x,J,\lambda_{n})\Big) = 0$ for any $x > 0$ and $J \geq 2$.
If $J$ remains fixed, Theorem  \ref{th:shift} thus  implies that $\hat\btheta^{\lambda}$ converges in probability to  $\btheta_{\bTheta}^*$ as  $n \to + \infty$. To the contrary, let us fix $n$, and consider an asymptotic setting where only $J  \to + \infty$. For any $x > 0$ and $\lambda \in \Lambda$,  
$
\lim_{J\to +\infty} \upsilon(x,J,\lambda) = V(\lambda).
$
Therefore, Theorem \ref{th:shift} cannot be used to prove that $\hat\btheta^{\lambda} $ converges to $\btheta_{\bTheta}^*$ as $J  \to + \infty$. This confirms that $\hat\btheta^{\lambda} $ is not a consistent estimator of  $\btheta_{\bTheta}^*$  (and thus of $\btheta^*$) as $n$ remains fixed and $J$ tends to infinity. 

\subsection{Consistent estimation of the mean pattern}

Recall that the estimator $ \hat f^{\lambda} $ of the mean pattern $f$ is defined as
$
 \hat f^{\lambda} = \frac{1}{J} \sum_{j=1}^J \tilde T_{ \hat\btheta^{\lambda}_{j}}\hat f^{\lambda}_j.
$
We study the consistency of $ \hat f^{\lambda} $ with respect to the shape function
$$
f_{\bTheta}^* :=  \frac{1}{J}\sum_{j=1}^J \tilde T_{[\btheta_{\bTheta}^*]_j} T_{\btheta_j^*}f,
$$
defined for $\btheta_{\bTheta}^* = ([\btheta_{\bTheta}^*]_1,\ldots,[\btheta_{\bTheta}^*]_J)$. Again, depending on the problem at hand and the choice of  the constrained set $\bTheta $, it can be shown that $f_{\bTheta}^*$ is close to the true mean pattern $f$. For example, in the case of shifted curves with $\bTheta = \bTheta_{0}$ defined in (\ref{eq:Theta0}), then $ \btheta_{\bTheta_0} = (\btheta^*_1 -\bar \btheta^*,\ldots,\btheta_{J}^* - \bar\btheta^*)$ with $  \bar\btheta^* = \frac{1}{J}\sum_{j=1}^J \btheta_j^*$. In this case
$
f_{\bTheta_0}^*(t) :=  \frac{1}{J}\sum_{j=1}^J f(t - \btheta_j^*+ [\btheta_{\bTheta_0}^*]_j) = f(t-\bar \btheta^*).
$
Hence, under the condition that $\int_{\Theta} \btheta g(\btheta) d \btheta = 0$, then  $\bar\btheta^* \approx 0$ for $J$ sufficiently large, and thus $f_{\bTheta}^*(t)$ is close to $f$ which allows to show the consistency of $ \hat f^{\lambda} $ to $f$ as formulated in Theorem \ref{theo:shift2}.

\begin{theo}\label{th:funct}
Consider the model \eqref{eq:modeldeform}  and  suppose that Assumptions   \ref{hyp:Theta}, \ref{hyp:unic}, \ref{hyp:ident} and \ref{ass:op12} to \ref{hyp:EigZmax} hold. Then, for any $\lambda \in \Lambda$ and $x > 0$ 
\begin{align}
\P\bigg(  \lVert  \hat f^{\lambda} - f_{\bTheta}^* \rVert^2_{L^2} \geq &  C_{2}(\Theta,\bTheta,\F,f) \Big[ (\gamma_{n}(\Theta)+\sigma^2)\Big( \sqrt{\upsilon(x,J,\lambda)} + \upsilon(x,J,\lambda)\Big) \nonumber \\
& + \Big(\sqrt{B(\lambda)} +B(\lambda)   \Big)\Big]   \bigg) \leq 2e^{-x}, \label{eq:dev2}
\end{align}
where $C_{2}(\Theta,\bTheta,\F,f)>0$ is a constant depending only $\Theta$, $\bTheta$, $\F$, and $f$.  
\end{theo}

The consistency of $\hat f^\lambda$ to $ f_{\bTheta}^*$ is thus guaranteed when $n$ goes to infinity provided the level of smoothing $\lambda=\lambda_n $ is chosen so that $\lim_{n\to +\infty} V(\lambda_n) = \lim_{n\to +\infty} B(\lambda_n) =0$. Again, if $n$ remains fixed and only $J$ is let going to infinity then Theorem \ref{th:funct}  cannot be used to prove the convergence of  $\hat f^\lambda$ to $ f_{\bTheta}^*$.

\section{Numerical experiments for randomly shifted curves} \label{sec:simus}

Consider the model \eqref{eq:modrandshift} with random shifts $\btheta_j$ having a uniform density $g$  with compact support equal to $[-\frac{1}{5},\frac{1}{5}]$, and $f (t)=9\sin(2\pi t) + 2\cos(8\pi t)$ for  $t\in[0,1]$ as a mean pattern, see Figure \ref{fig:data.1}. For the constrained set we took
$$
\bTheta_0 = \left\{ \btheta \in \left[-\tfrac{1}{2},\tfrac{1}{2}\right]^J,\ \btheta_1 + \cdots + \btheta_J = 0 \right\}.
$$
We use Fourier low pass filtering with spectral  cut-off to $\lambda  =7$ which is reasonable value to reconstruct $f$ representing a good tradeoff between bias and variance. We present some results of simulations under various assumptions of the process $Z$ and the level $\sigma$ of additive noise in the measurements. \\
 \begin{figure}[!ht]
\begin{center}
\subfigure[]{\includegraphics[height=4cm]{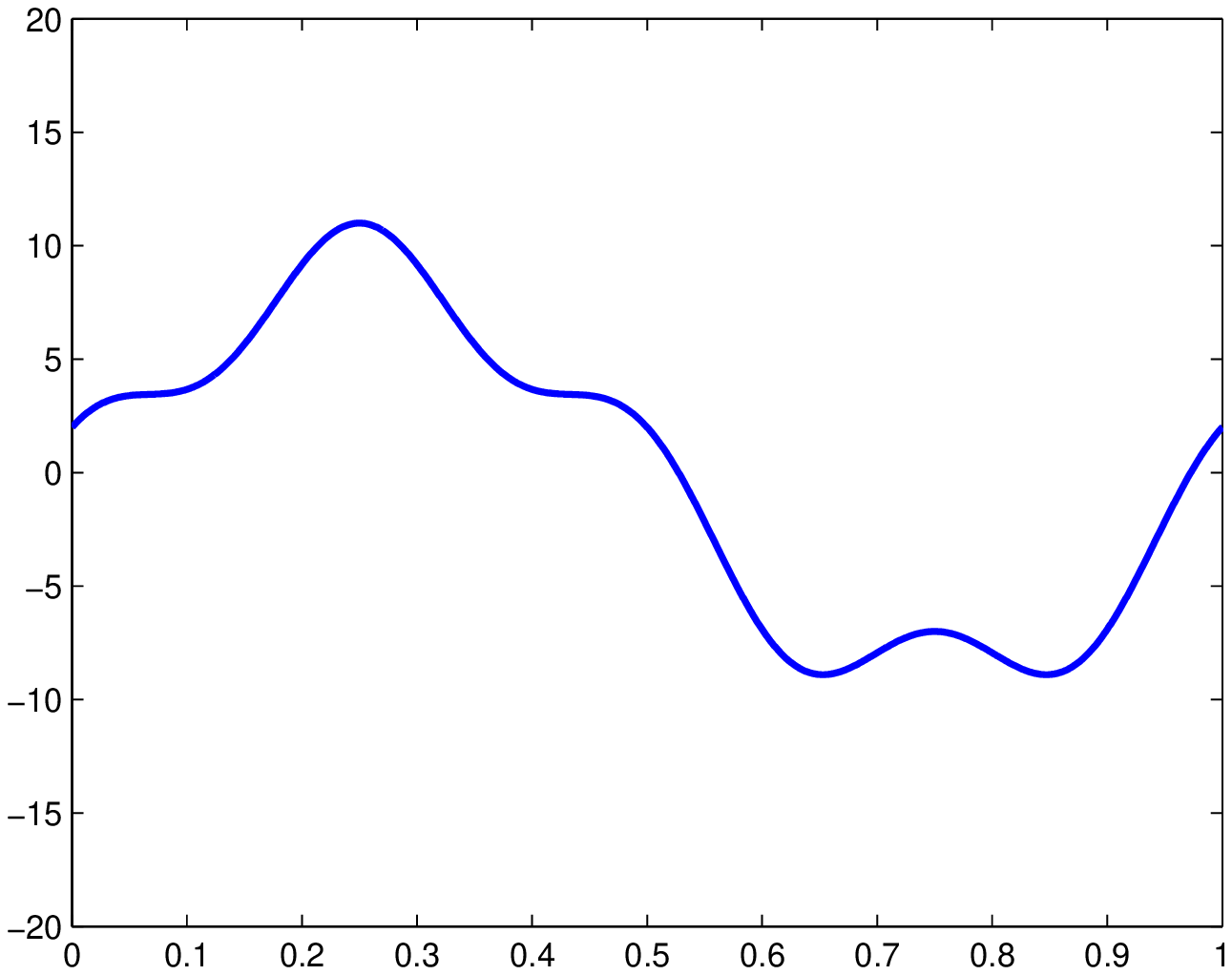}\label{fig:data.1}}
\subfigure[][]{\includegraphics[height=4cm]{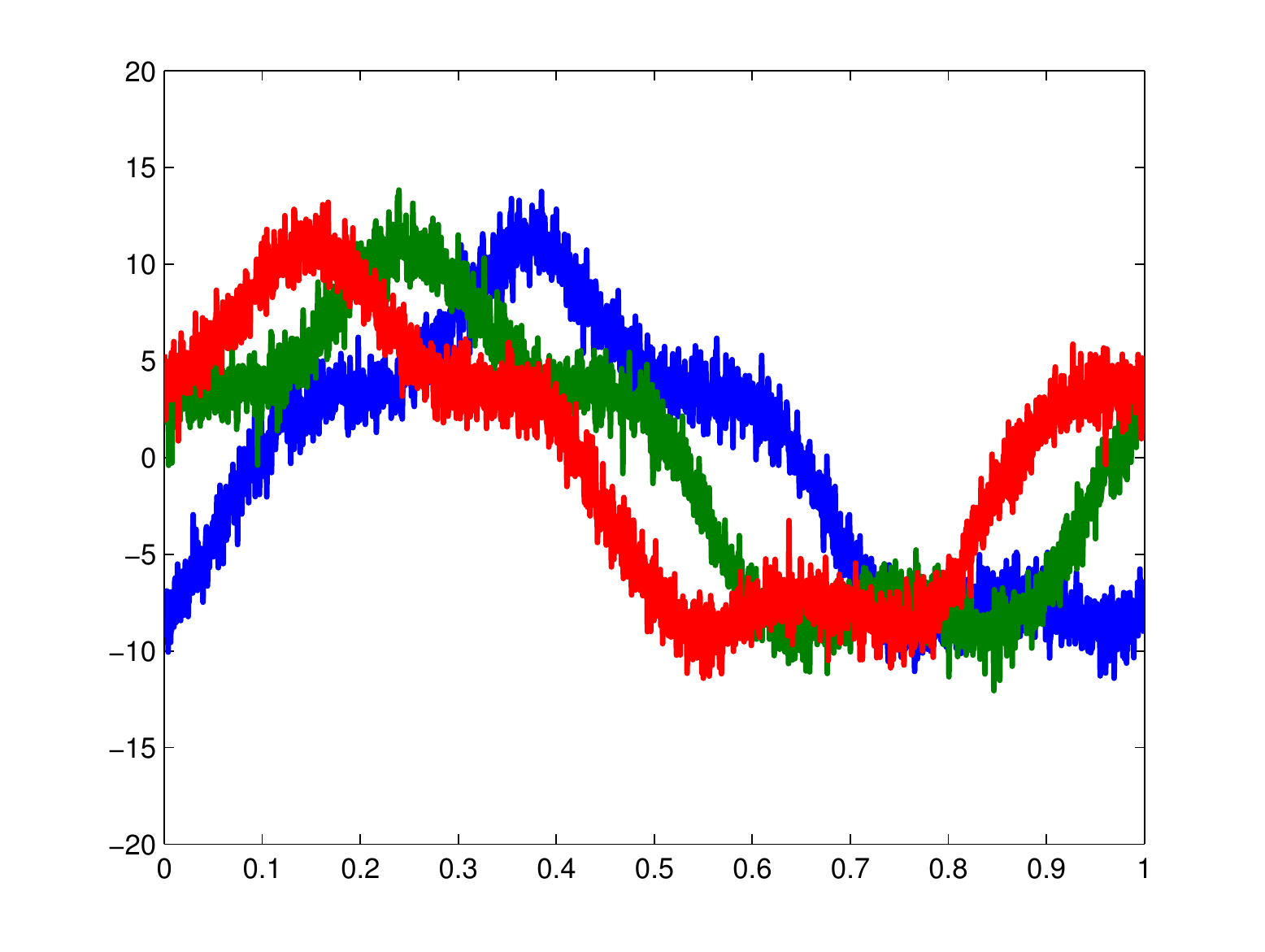}\label{fig:data.2}} 
\subfigure[][]{\includegraphics[height=4cm]{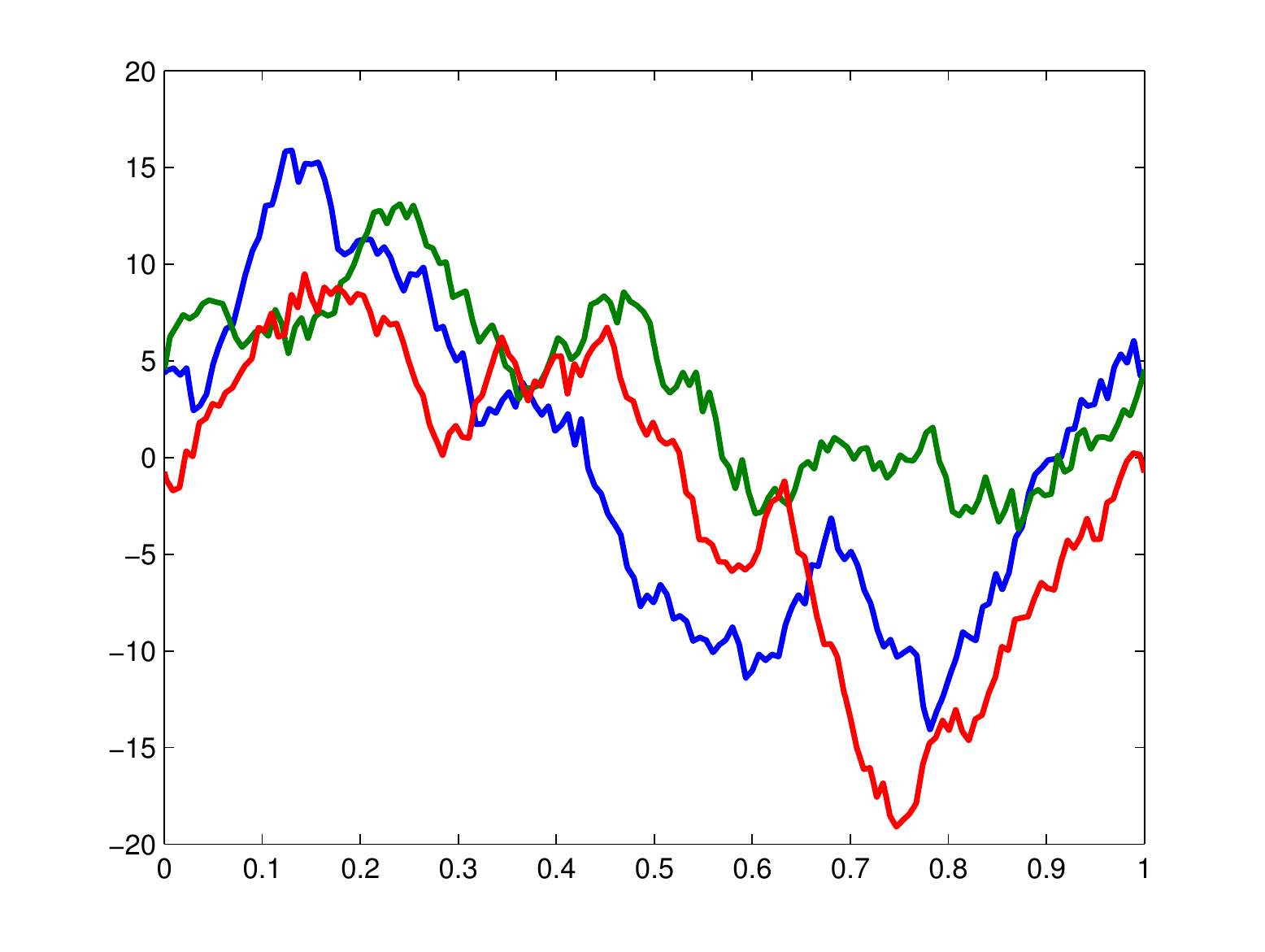}\label{fig:data.3}}
\caption{\subref{fig:data.1} mean pattern $f$. \subref{fig:data.2} $J=3$ noisy curves in the SIM with $\sigma=2$.  \subref{fig:data.3} $J=3$ noisy curves with $\sigma=0$ and a stationary process $Z$ with $\varsigma = 4$. }
\end{center}
\end{figure}

\noindent {\bf Shape invariant model (SIM).}  The first numerical applications illustrate the role of $n$ and $J$ in the SIM model. Figure \ref{fig:data.2} gives a sample of the data used with $\sigma = 2$.  The factors in the simulations are the number $J$ of curves and the number of design points $n$. For each combination of these two factors, we simulate $M = 20$ repetitions of model \eqref{eq:modrandshift}. For each repetition we computed  $\frac{1}{J}\|\hat\btheta^\lambda - \btheta^*\|^2$ and $\|\hat f^\lambda - f\|^2_{L^2}$. Boxplot of these quantities are displayed in Figure \ref{fig:lb.1} and \ref{fig:lb.2} respectively, for  $J= 20,40,\ldots,100$ and  $n=512$ (in gray) and  $n=1024$  (in black). As the smoothing parameter is fixed to $\lambda = 7$, increasing $n$ simply reduces the variance of the linear smoothers $\hat{f}^{\lambda}_{j}$. Recall that the lower bound given in Theorem \ref{theo:vantreeshift} shows that   $\frac{1}{J}\Ec[\|\btheta^* - \hat\btheta^\lambda\|^2]$ does not decrease as $J$ increases but should be smaller when the number of point $n$ increases. This is exactly what we observe in Figure \ref{fig:lb}. Similarly, the quantity $\|\hat f^\lambda - f\|^2_{L^2}$  is clearly smaller with $n=1024$ than with $n=512$.


\begin{figure}[!ht]
\begin{center}
\subfigure[][]{\includegraphics[height=4cm]{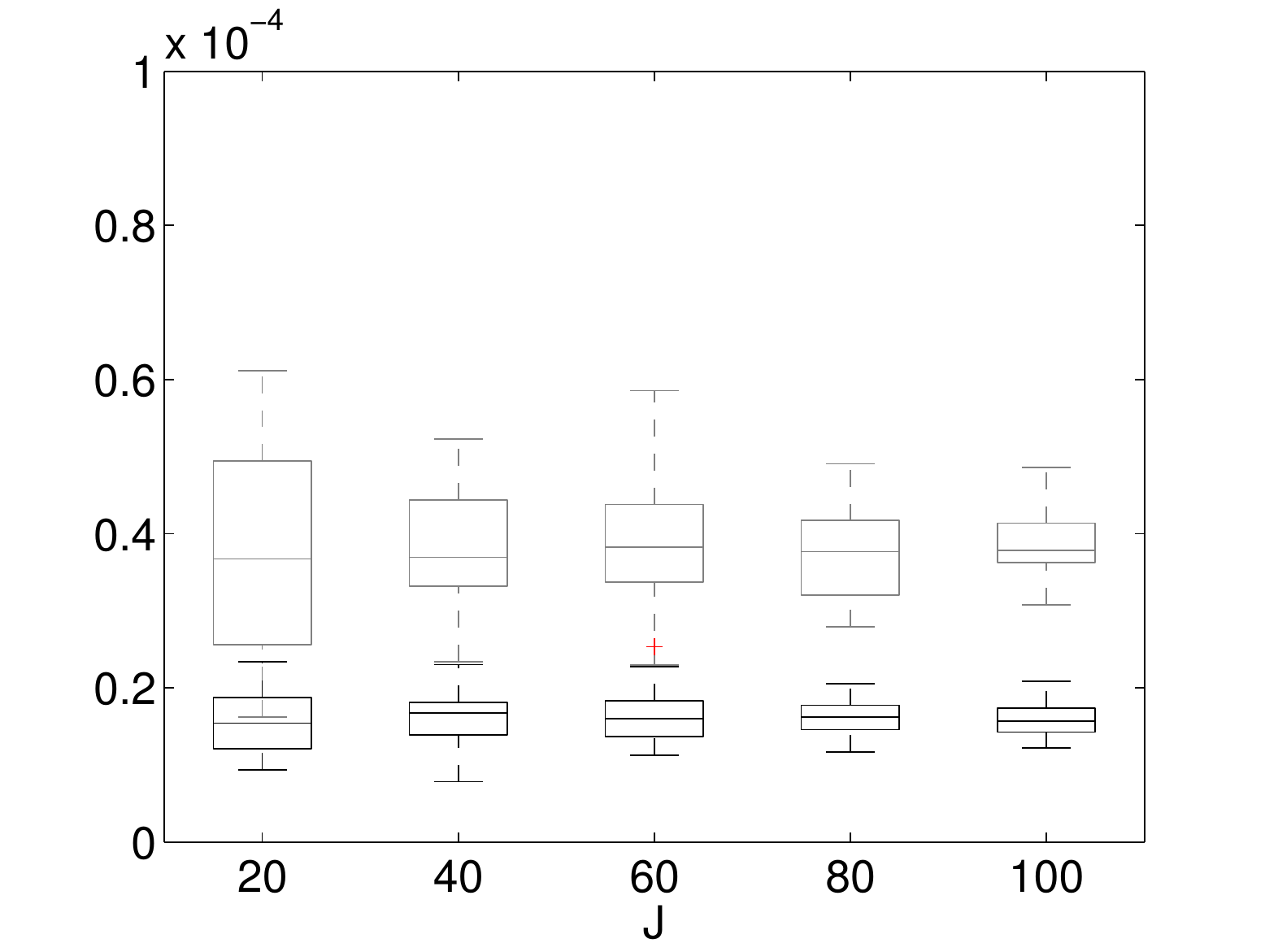}\label{fig:lb.1}}
 \subfigure[][]{\includegraphics[height=4cm]{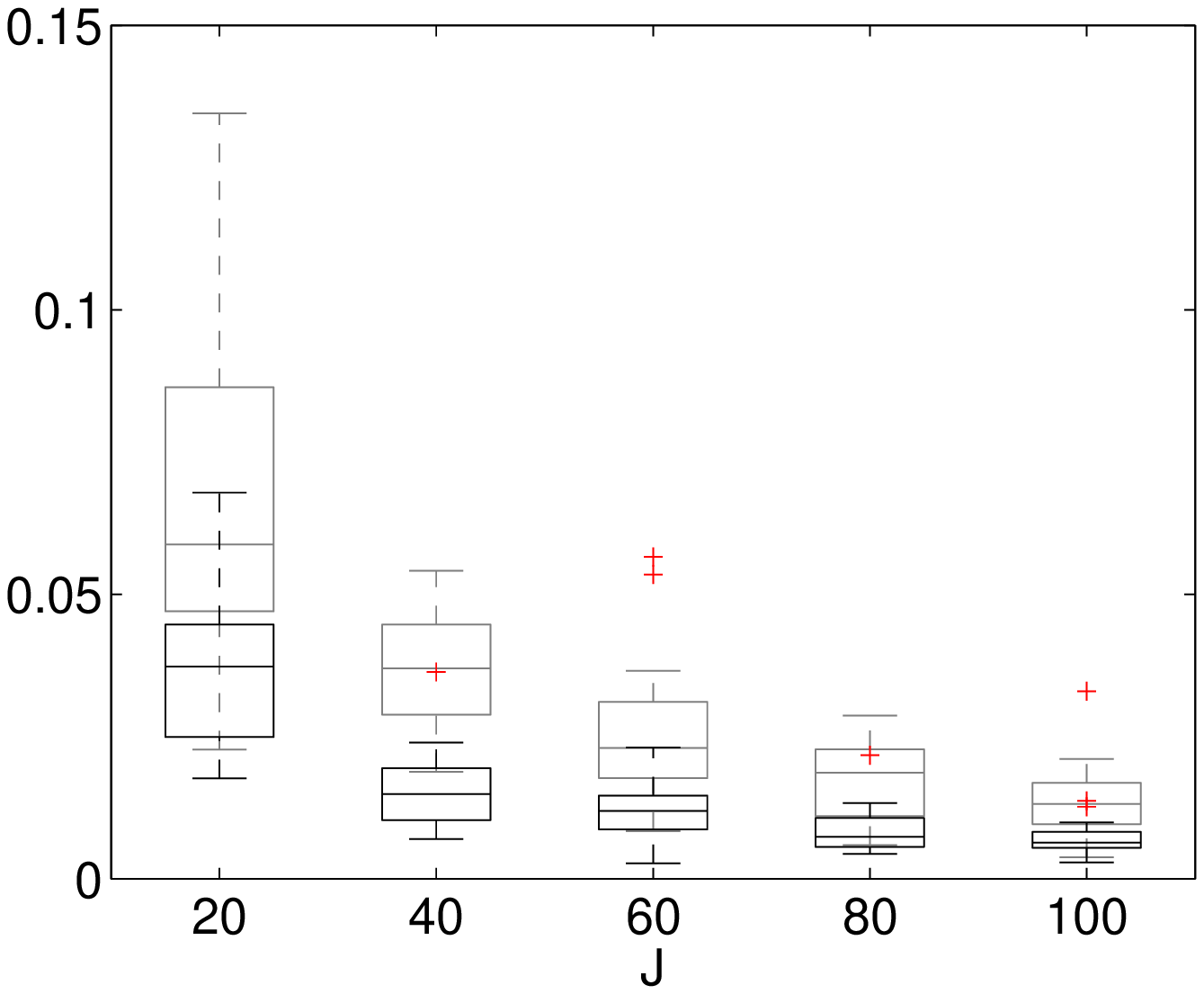}\label{fig:lb.2}}
\caption{Boxplot of  $\frac{1}{J}\|\hat\btheta^\lambda - \btheta^*_{\bTheta_0} \|^2$ \subref{fig:lb.1} and $\|\hat f^\lambda-f_{\bTheta_0}\|^2_{L^2}$ \subref{fig:lb.2} over $M= 20$ repetitions from a SIM model of shifted curves. Boxplot  in gray correspond to $n=512$, and in black to $n=1024$. }\label{fig:lb}
\end{center}
\end{figure}

\noindent{\bf Complete model.}  We now add the terms $Z_j$ in \eqref{eq:modrandshift} to model linear variations in amplitude of the curves around the template $f$. First, we generate a stationary periodic Gaussian process. To do this, the covariance matrix must be a particular Toeplitz matrix. 
As suggested in \cite{Gre} one possibility is to choose 
$$
R(t)= \varsigma^2 \frac{e^{\phi(t -1/2)}+e^{-\phi(t -1/2)}}{e^{\phi/2} + e^{-\phi/2}},
$$ 
where $\phi$ is a strictly positive parameter (we took $\phi = 4$) and $\varsigma$ a variance parameter. The level of additive noise is $\sigma = 8$, and we took $\varsigma=4$.  As an illustration, in Figure \ref{fig:data.3} we plot $f+Z_j$, $j=1,2,3$ with $\varsigma=\phi = 4$. Over $M=20$ repetitions, we have computed the values of $ \frac{1}{J}\|\hat\btheta^\lambda - \btheta^*_{\bTheta_0} \|^2$ and $ \|\hat f^\lambda-f_{\bTheta_0}\|^2_{L^2}$ for $J$ is varying from $20$ to $100$ and $n=512,1024$. The results are displayed in Figure \ref{fig:lbZ.1} and \ref{fig:lbZ.2}. We observe the same behaviors than in the simulations with the SIM model: the variance of $\frac{1}{J}\|\hat\btheta^\lambda - \btheta^*_{\bTheta_0} \|^2 $ does not decrease as $J$ increases (see Figure \ref{fig:lbZ.1}) and $\|\hat f^\lambda-f_{\bTheta_0}\|^2$ has a smaller mean and variance as $n$ increases.
\begin{figure}[!ht]
\begin{center}
\subfigure[]{\includegraphics[height=4cm]{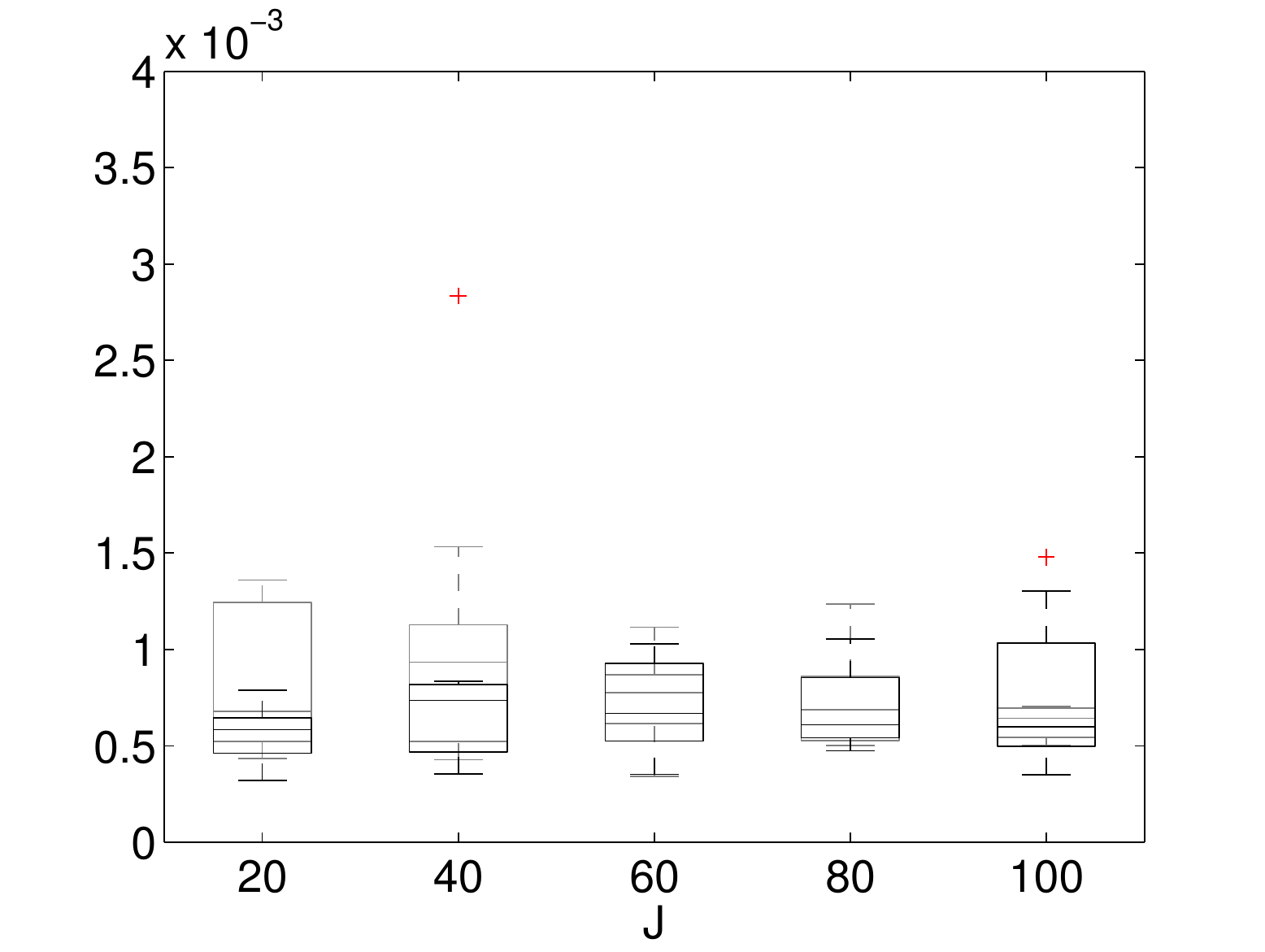}\label{fig:lbZ.1}}
\subfigure[]{\includegraphics[height=4cm]{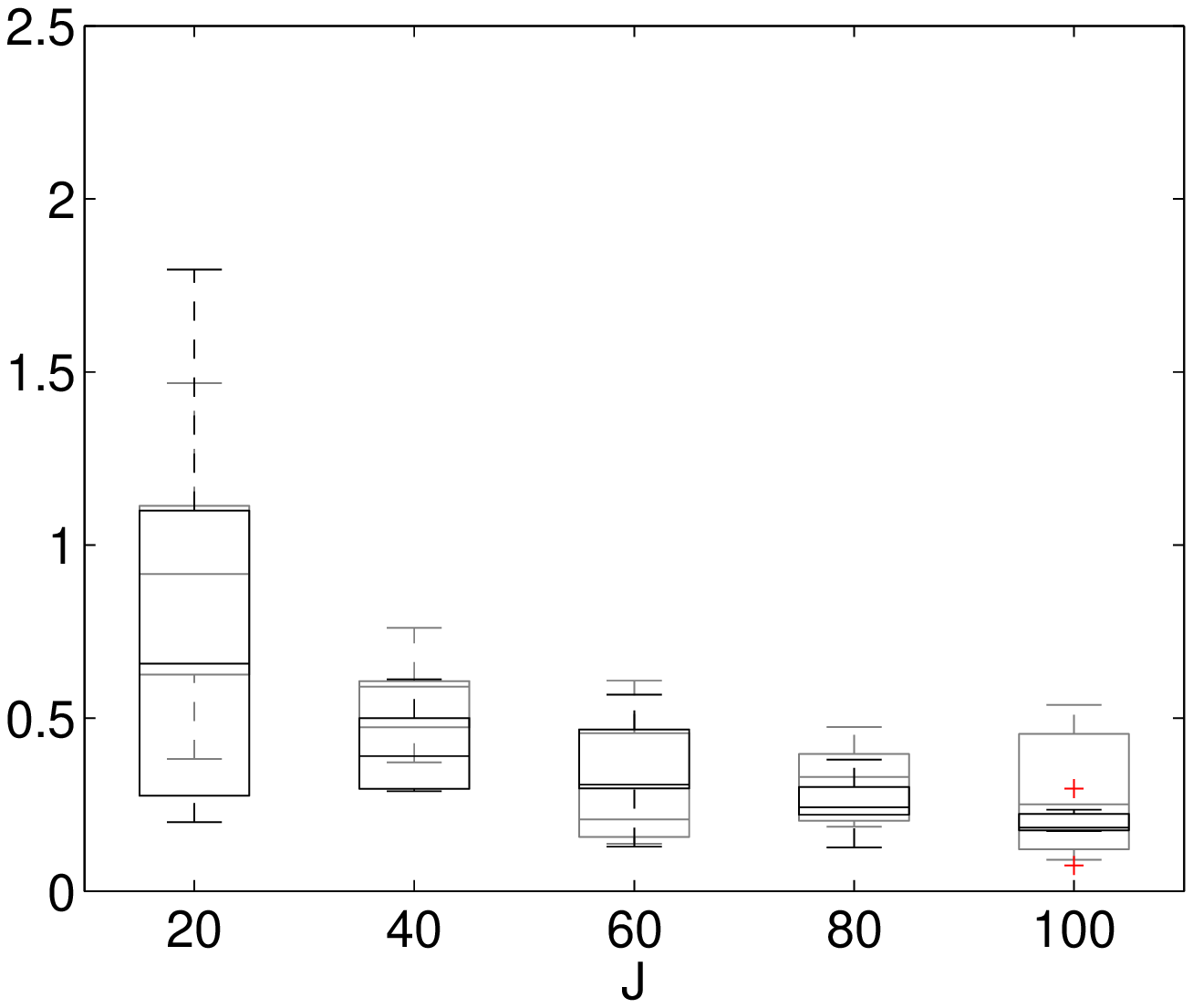}\label{fig:lbZ.2}}
\caption{Boxplot of   $\frac{1}{J}\|\hat\btheta^\lambda - \btheta^*_{\bTheta_0} \|^2$ \subref{fig:lbZ.1}  and  $\frac{1}{J}\|\hat f^\lambda-f_{\bTheta_0}\|^2$ \subref{fig:lbZ.2} in model \eqref{eq:modrandshift} with a stationnary error term $Z$.  Boxplot  in gray correspond to $n=512$, and in black to $n=1024$.}\label{fig:lbZ}
\end{center}
\end{figure}
 
We finally run the same simulations with a non stationary noise $Z_j(t) = \alpha_j \psi(t)$ where $\psi$ is a positive periodic smooth deterministic function such that $\norm{\psi}_{L^2} =1$ and $\alpha_j \sim \mathcal{N}(0,\varsigma^2)$ with $\varsigma = 4$. Note that, in this case, the sequence $\gamma_n(\Theta)$ is of order $n$ and Assumption \ref{hyp:EigZmax} is not verified. The levels of  noise ($\sigma$ and $\varsigma$) are the same than in the stationary case in order to make things comparable. The results are presented in the same manner in  Figure \ref{fig:lbZmod.1} for $ \frac{1}{J}\|\hat\btheta^\lambda - \btheta^*_{\bTheta_0} \|^2$ and in Figure \ref{fig:lbZmod.2} for $\|\hat f^\lambda-f_{\bTheta_0}\|^2_{L^2}$. One can see that the results are very different. The estimators of the shifts have a much larger mean and variance, and the variance  of$\frac{1}{J}\|\hat\btheta^\lambda - \btheta^*_{\bTheta_0} \|^2$  remains rather high even when $n$ or $J$ increases (see Figure \ref{fig:lbZmod.1}). The convergence to zero of $\|\hat f^\lambda-f_{\bTheta_0}\|^2_{L^2}$ which was clear in the stationary case, is now not so obvious in view of the numerical results displayed in Figure \ref{fig:lbZmod.2}. 
\begin{figure}[!ht]
\begin{center}
\subfigure[]{\includegraphics[height=4cm]{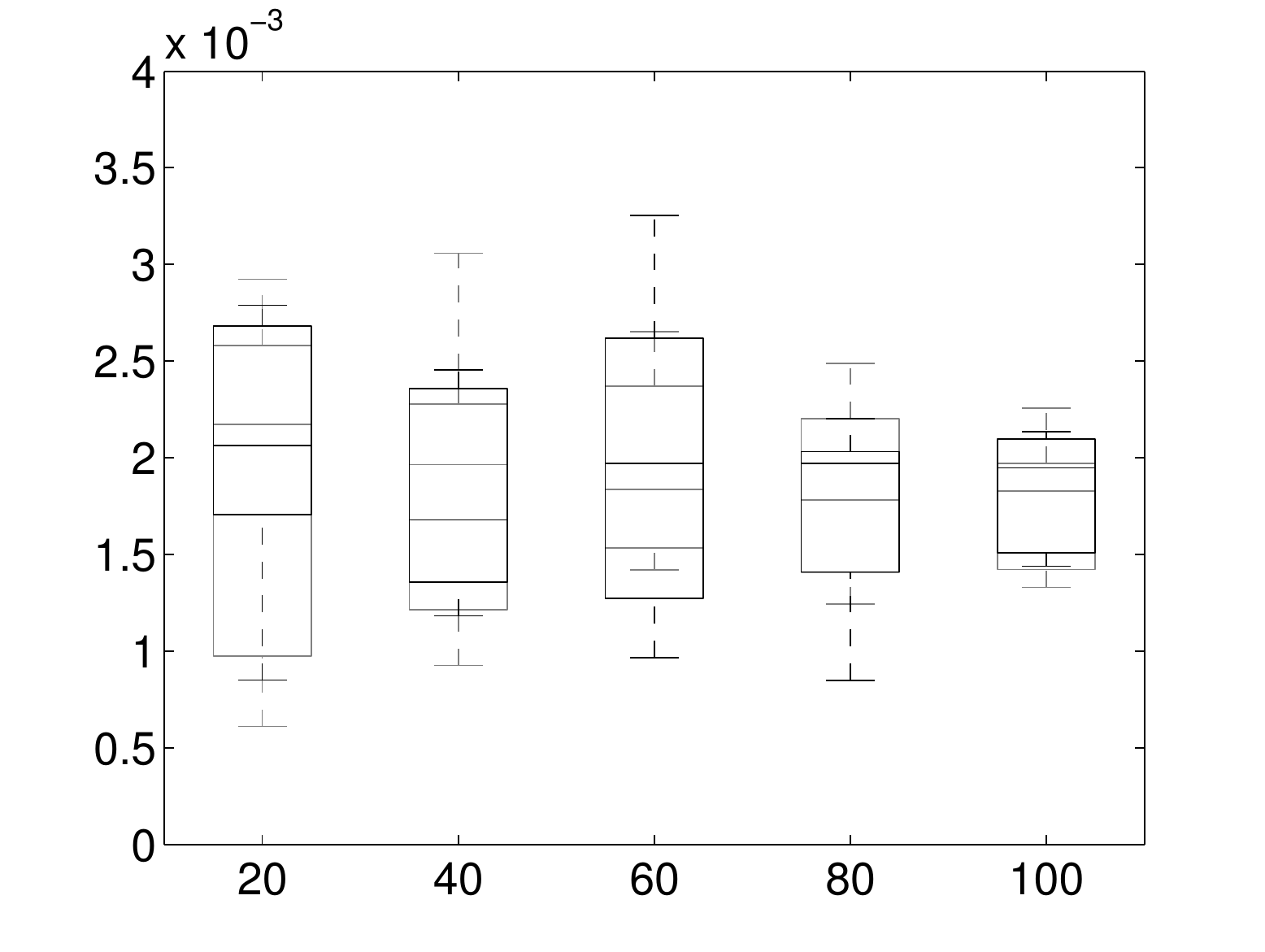}\label{fig:lbZmod.1}}
 \subfigure[]{\includegraphics[height=4cm]{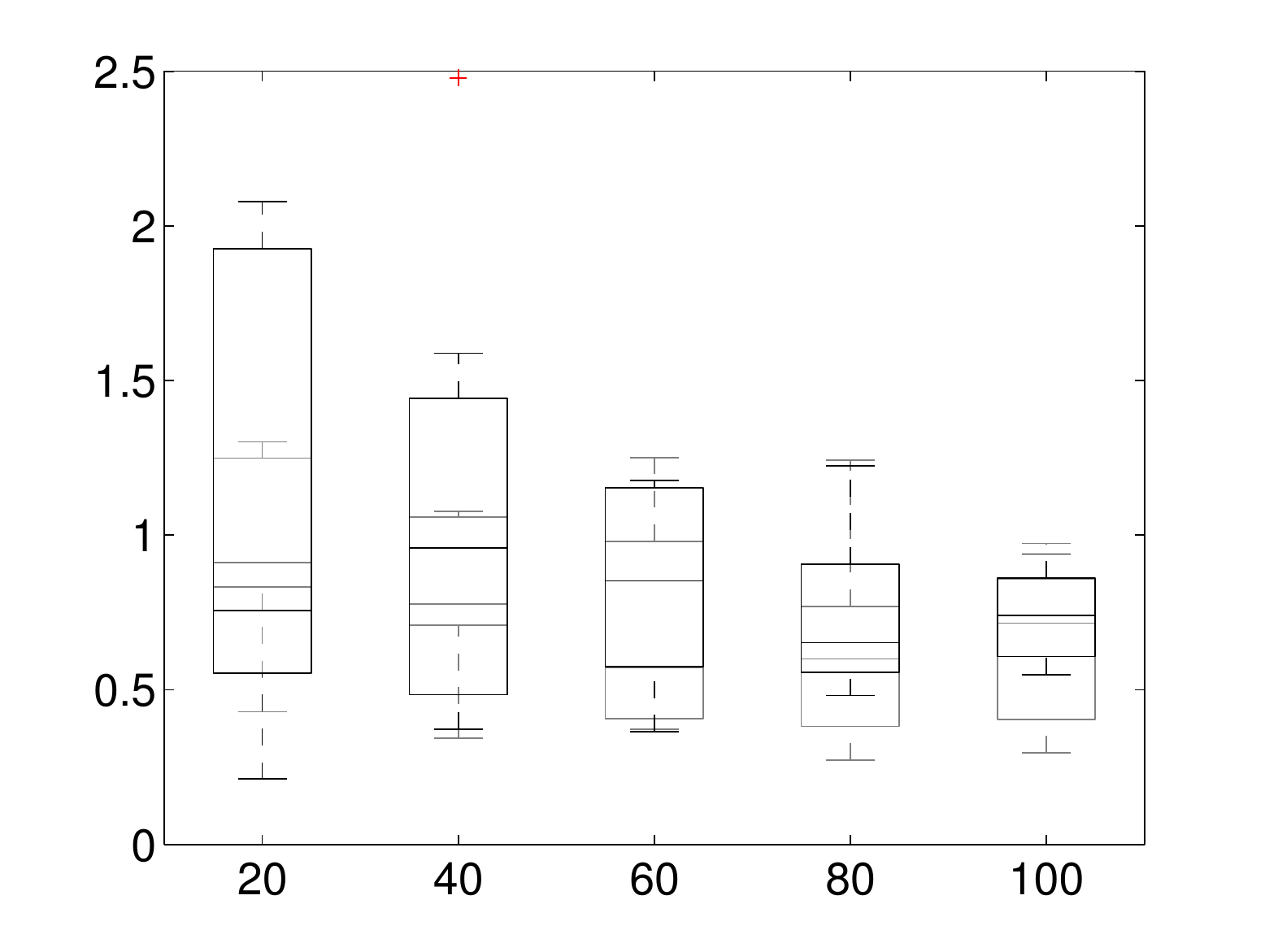}\label{fig:lbZmod.2}}
\caption{Boxplot of   $\frac{1}{J}\|\hat\btheta^\lambda - \btheta^*_{\bTheta_0} \|^2$ \subref{fig:lbZmod.1}  and  $\frac{1}{J}\|\hat f^\lambda-f_{\bTheta_0}\|^2$ \subref{fig:lbZmod.2} in model \eqref{eq:modrandshift} with a non-stationnary error term $Z$.  Boxplot  in gray correspond to $n=512$, and in black to $n=1024$.\label{fig:lbZmod}}
\end{center}
\end{figure}

\section{Conclusion and perspectives} \label{sec:conc}

We have proposed to use a Fr\'echet mean of smoothed data to estimate a mean pattern of curves or images satisfying a non-parametric regression model including random deformations. 
 Upper and lower bounds (in probability and expectation)  for the estimation of the deformation parameters and the mean pattern have been derived. Our main result is that these bounds go to zero as the dimension $n$ of the data (the number of sample points) goes to infinity, but that an asymptotic setting only in $J$ (the number of observed curves or images) is not sufficient to obtain consistent estimators. An interesting topic for future investigation would be to study the rate of convergence of such estimators and to analyze their optimality (e.g.\ from a minimax point of view).




\begin{appendix}

\section{Proof of the results in Section \ref{part:lower}}

\subsection{Proof of Theorem \ref{theo:VTSIM}}\label{part:proofVT}

Write $\btheta^*_j = ([\theta^{*}]^{1}_j,\ldots,[\theta^{*}]^{p}_j)$, and let $\Y =(\Y_1,\ldots,\Y_J) \in \R^{nJ}$ be the column vector of the observations generated by model \eqref{eq:SIM}. Conditionally to $\btheta^*$, $\Y$ is a Gaussian vector and  its log-likelihood is equal to
\begin{align}\label{eq:mat0}
\log(p(\Y\vert \btheta^*)) = -\frac{Jn}{2}\log(2\pi) + \frac{J}{2} \log(\det(\Lambda)) -\frac{1}{2} \sum_{j=1}^J(\Y_j-\T_{\btheta^*_{j}}\f) ' \Lambda(\Y_j-\T_{\btheta_j^*}\f),
\end{align}
where $\Lambda = \sigma^{-2} Id_{n}$.
Therefore, we have the expected score $\E_{\btheta^*}[\partial_{[\theta^{*}]^{p_1}_{j_1}} \log(p(\Y\vert \btheta^*))] = 0$ for all $j_1 = 1,\ldots,J$ and $p_1 =1,\ldots,p$ and
\begin{align}\label{eq:mat2}
\E_{\btheta^*}\big[\partial_{[\theta^{*}]^{p_1}_{j_1}} \log(p(\Y\vert \btheta^*))\partial_{[\theta^{*}]^{p_2}_{j_2}} \log(p(\Y\vert \btheta^*)) \big] = \begin{cases}
						 0  & \text{ if } j_1 \neq j_2,\\
						-\big [(\partial_{[\theta^{*}]^{p_1}_{j_1} }  \T_{\btheta_{j_1}^*}\f )'\ \Lambda \ (\partial_{[\theta^{*}]^{p_2}_{j_1}} \T_{\btheta_{j_1}^*}\f)\big]_{p_1,p_2 =1}^p & \text{ if } j_1 = j_2,
						 \end{cases}
\end{align}
where $\partial_{[\theta^{*}]^{p_1}_{j_1}} \T_{\btheta_{j_1}^*}\f  =\big[\partial_{[\theta^{*}]^{p_1}_{j_1}} T_{\btheta_{j_1}^*}f(t_\ell)  \big]_{\ell=1}^n$. Then, for each $j_1 = 1,\ldots,J$ and $p_1 = 1,\ldots,p$ we have
\begin{align}\label{eq:mat3}
(\partial_{[\theta^{*}]^{p_1}_{j_1} }  \T_{\btheta_{j_1}^*}\f )'\ \Lambda \  (\partial_{[\theta^{*}]^{p_1}_{j_1}} \T_{\btheta_{j_1}^*}\f) \leq \sigma^{-2} \snorm{\partial_{[\theta^{*}]^{p_1}_{j_1} }  \T_{\btheta_{j_1}^*}\f}^2 \leq C(\Theta,f)n \sigma^{-2},
\end{align}
where the last inequality is a consequence of Assumption \ref{hyp:op3}. From now on, $\hat\btheta =\hat\btheta(\Y) =(\hat\btheta_1(\Y),\ldots,\hat\btheta_1(\Y))$ is an arbitrary estimator (i.e any measurable function of $\Y$) of the true parameter $\btheta^*$. Let also 
$$
U = \hat \btheta -\btheta^* \quad \text{ and } \quad V = \Big[ [\partial_{[\theta^{*}]^{p_1}_{1}} \log(p(\Y\vert \btheta^*) g(\btheta^*))]_{p_1=1}^p ,\ldots, [\partial_{[\theta^{*}]^{p_1}_{J}} \log(p(\Y\vert \btheta^*) g(\btheta^*))]_{p_1=1}^p\Big] 
$$ 
be a matrix of column vectors of $\R^{pJ}$. Then, Cauchy-Schwarz inequality implies
\begin{equation}\label{eq:CSc}
(\E[U' V])^2 \leq \E[U'U] \E[V' V].
\end{equation}   
In the sequel we note $g^J(\btheta)d\btheta = g(\btheta_1)\ldots g(\btheta_J)d\btheta_1\ldots d\btheta _J$. We have
\begin{align*}
\E[U' V]&  = \sum_{j=1}^J\sum_{p_1=1}^p \int_{\R^{nJ}} \int_{\Theta^J} (\hat\theta_j^{p_1}(y)-[\theta ]_j^{p_1}) \partial_{[\theta^{*}]^{p_1}_{j}} (p(y\vert \btheta )  g^J(\btheta )) d\btheta dy\\
& =\sum_{j=1}^J\sum_{p_1=1}^p \int_{\R^{nJ}} \hat\theta_j^{p_1}(y) \int_{\Theta^J} \partial_{[\theta^{*}]^{p_1}_{j}} (p(y\vert \btheta ) g^J(\btheta )) d\btheta  dy \\
& \qquad -\sum_{j=1}^J\sum_{p_1=1}^p \int_{\R^{nJ}} \int_{\Theta^J} [\theta ]^{p_1}_j \partial_{[\theta^{*}]^{p_1}_{j}} (p(y\vert \btheta ) g^J(\btheta )) d\btheta dy
\end{align*}
Assumption \ref{hyp:Theta} and the differentiability of $g$ imply that for all $p_1 = 1,\ldots,p $ and all $\btheta\in\Theta$ we have $\lim_{\theta^{p_1}\to \rho} g(\btheta) =0$. Then, an integration by part and Fubini's theorem give $\int_{\Theta^J} \partial_{[\theta^{*}]^{p_1}_{j}} (p(y\vert \btheta ) g^J(\btheta )) d\btheta  =0 $. Again, with the same arguments,
$\int_{\Theta^J} [\theta ]^{p_1}_j \partial_{[\theta^{*}]^{p_1}_{j}} (p(y\vert \btheta ) g^J(\btheta )) d\btheta  = -\int_{\Theta^J} p(y\vert \btheta ) g^J(\btheta ) d\btheta   $ and thus
$
\E[U' V] = pJ.
$

Now, using that the expected score is zero and equation \eqref{eq:mat2} we have 
\begin{align*}
\E[V'V] & =\sum_{j=1}^J\sum_{p_1=1}^p  \E[(\partial_{[\theta^{*}]^{p_1}_{j}} \log(p(\Y\vert \btheta^*))^2] +\E[(\partial_{[\theta^{*}]^{p_1}_{j}}  \log(g(\btheta^*))^2]\\
& = \sum_{j=1}^J\sum_{p_1=1}^p \int_{\Theta^J}(\partial_{[\theta]^{p_1}_{j} }  \T_{\btheta_{j}}\f )'\ \Lambda \  (\partial_{[\theta^{*}]^{p_1}_{j}} \T_{\btheta_{j_1}}\f)g^J(\btheta)d\btheta + J\int_{\Theta}\left\| \partial_{\btheta_1} \log\left(g(\btheta_1)\right)\right\|^2 g(\btheta_1)d\btheta_1. 
\end{align*}
where $\partial_{\btheta_1} \log\left(g(\btheta_1)\right) = [\partial_{[\btheta]^{1}_1} \log\left(g(\btheta_1)\right) ,\ldots,\partial_{[\btheta]^{p}_1} \log\left(g(\btheta_1)\right)] \in\R^p$. Then, using inequality \ref{eq:mat3}, it gives
$
\E[V'V] \leq pJ n C(\Theta,f)\sigma^{-2} + J \int_{\Theta}\left\|\partial_{\btheta_1} \log\left(g(\btheta_1)\right)\right\|^2 g(\btheta_1)d\btheta_1.
$
 Hence, using equation \eqref{eq:CSc} for any estimator $\hat\btheta =\hat\btheta(\Y)$ (see Theorem 1 in \cite{MR1354456})
\begin{align*}
\E\left[\|\hat\btheta - \btheta^*\|^2 \right] 
& \geq \frac{pJ}{ n C(\Theta,f)\sigma^{-2} + p^{-1}\int_{\Theta}\left\|\partial_{\btheta_1} \log\left(g(\btheta_1)\right)\right\|^2 g(\btheta_1)d\btheta_1}\\
& \geq \frac{\sigma^{2}n^{-1}pJ}{C(\Theta,f) + n^{-1}p^{-1}\sigma^{2} \int_{\Theta}\left\|\partial_{\btheta_1} \log\left(g(\btheta_1)\right)\right\|^2 g(\btheta_1)d\btheta_1}.
\end{align*}
And since $p\geq 1 $, the claim in  Theorem \ref{theo:VTSIM} is proved. \hfill $\Box$

\subsection{Proof of Theorem \ref{theo:VTGEN}}

 As above, let $\Y\in \R^{nJ}$ is the column vector generated by model \eqref{eq:modeldeform}. Then, conditionally to $\btheta^*$, $\Y$ is a Gaussian vectors and  Assumption \ref{hyp:Zmin} ensures that its log-likelihood has the same expression as in equation \eqref{eq:mat0} but with 
$$
\Lambda =\Lambda(\Theta)  = (\sigma^2 Id_{n} +\Ec_{\btheta^*}\big[ \T_{\btheta^*_j}\bZ_j ( \T_{\btheta^*_j}\bZ_j)'\big])^{-1} = (\sigma^2 Id_{n} +\mathbf \Sigma_n(\Theta) )^{-1}
$$
As the matrix $\mathbf \Sigma_n(\Theta)$ is positive semi definite with it smallest eigenvalue denoted by $s_n^2(\Theta)$ (see Assumption \ref{hyp:Zmin}), the uniform bound \eqref{eq:mat3} becomes 
$$
(\partial_{[\theta^{*}]^{p_1}_{j_1} }  \T_{\btheta_{j_1}^*}\f )'\ \Lambda(\Theta) \  (\partial_{[\theta^{*}]^{p_1}_{j_1}} \T_{\btheta_{j_1}^*}\f) 
 \leq (\sigma^{2} + s_n^2(\Theta))^{-1} \snorm{\partial_{[\theta^{*}]^{p_1}_{j_1} }  \T_{\btheta_{j_1}^*}\f}^2 \leq C(\Theta,f)n (\sigma^{2}+ s^2_n(\Theta))^{-1},
$$ 
for all $p_1=1,\ldots,p$ and $j=1,\ldots,J$. As above the last inequality is a consequence of Assumption \ref{hyp:op3} and the rest of the proof is identical to the proof of Theorem \ref{theo:VTSIM}. \hfill $\Box$

\subsection{Proof of Theorem \ref{theo:vantreeshift}} 

For all $\btheta\in \R$ the operators $T_{\btheta}f(\cdot) = f(\cdot - \btheta)$ are isometric from $L^2([0,1])$ to $L^2([0,1])$ as a change of variable implies immediately that $\norm{T_{\btheta}f}^2_{L^2} =\norm{f}^2_{L^2}$. For all continuously differentiable function $f$, we have $ \partial_{\theta} T_{\btheta}f(t) = -sign(\btheta) \partial_{t} f(t - \btheta),$ where 
$sign(\cdot)$ is the sign function. Then, for all $\btheta\in\Theta$, $ \snorm{\partial_{\theta} T_{\btheta}f}^2_{L^2} = \snorm{\partial_{t} f}^2_{L^2} \leq \snorm{\partial_{t} f}_{\infty}^2$ and Assumption \ref{hyp:op3} is satisfied with $C(\Theta,f) = \snorm{\partial_{t} f}_{\infty}^2$. Finally, as the error terms $Z_j$'s are i.i.d stationary random process the covariance function is invariant by the action of the shifts and Assumption \ref{hyp:Zmin} is satisfied with $\Sigma_n(\Theta) = \mathbf{\Sigma}_{n}$ defined in \eqref{eq:Sigman} (see Section \ref{part:Zshift} for further details). Then, the result of Theorem \ref{theo:vantreeshift} follows as an application of Theorem \ref{theo:VTGEN}.  \hfill $\Box$

\section{Proof of the results in Section \ref{sec:identgen}}

\subsection{Proof of Proposition \ref{prop:D}}

Remark that
$
D(\btheta) = \sum_{k \in \Z} | c_{k}^*|^2 \bigg(1 - \bigg|  \frac{1}{J} \sum_{j=1}^{J}  e^{i2\pi  k (\btheta_{j}-\btheta_{j}^*)}   \bigg|^2 \bigg),
$ 
where $c_{k}^{\ast} = \int_0^1 f(t) e^{-i2\pi k t}dt$. Thanks to Assumption \ref{hyp:f}, it follows that for any $\btheta \in \bTheta $,
\begin{align} \label{eq:loweboundD}
D(\btheta) \geq  | c_{1}^{\ast}|^2 \bigg(1 - \bigg|  \frac{1}{J} \sum_{j=1}^{J}  e^{i2\pi (\btheta_{j}-\btheta_{j}^*)}   \bigg|^2 \bigg)
\end{align}
with $c_{1}^{\ast} \neq 0 $. Then, remark that
$$
\bigg| \frac{1}{J} \sum_{j=1}^{J}  e^{i2\pi   (\btheta_{j}-\btheta_{j}^*)} \bigg|^2
= \frac{1}{J} + \frac{2}{J^2} \sum_{j=1}^{J-1} \sum_{j'=j+1}^{J} \cos\left(2\pi  \left( (\btheta_{j}-\btheta_{j}^*) -  (\btheta_{j'}-\btheta_{j'}^*) \right) \right).
$$
Using a second order Taylor expansion and the mean value theorem, one has that $\cos(2 \pi u) \leq 1 -  C(\rho) |u|^{2}$ for any real $u$ such that $|u|  \leq 4 \rho < 1/4$ with $C(\rho) = 2 \pi^2\cos(8 \pi \rho) > 0$. Therefore, the above equality implies that for  any $\btheta \in \bTheta $
\begin{align*}
\bigg| \frac{1}{J} \sum_{j=1}^{J}  e^{i2\pi   (\btheta_{j}-\btheta_{j}^*)} \bigg|^2
 &\leq \frac{1}{J} + \frac{2}{J^2} \sum_{j=1}^{J-1} \sum_{j'=j+1}^{J} 1 - C(\rho) \left| (\btheta_{j}-\btheta_{j}^*) -(\btheta_{j'}-\btheta_{j'}^*) \right|^2 \\
 &\leq 1 - \frac{2}{J^2} \sum_{j=1}^{J-1} \sum_{j'=j+1}^{J} C(\rho) \left| (\btheta_{j}-\btheta_{j}^*) -(\btheta_{j'}-\btheta_{j'}^*) \right|^2,
\end{align*}
since $ |(\btheta_{j}-\btheta_{j}^*) -(\btheta_{j'}-\btheta_{j'}^*)| \leq 4 \rho < 1/4$ for all $m,q=1,\ldots,n$  by Assumption \ref{hyp:g} and the hypothesis that $\rho < 1/16$. Hence, using the lower bound (\ref{eq:loweboundD}), it follows that for all $\btheta \in \bTheta $
\begin{equation}\label{ed:minD}
D(\btheta) \geq  C(f,\rho) \frac{1}{J^2}  \sum_{j=1}^{J-1} \sum_{j'=j+1}^{J} \left| (\btheta_{j}-\btheta_{j}^*) -(\btheta_{j'}-\btheta_{j'}^*) \right|^2 
\end{equation}
with $C(f,\rho) = 2 | c_{1}^{\ast}|^2 C(\rho)$. Now assume that $\btheta \in \bTheta_0$. Using the properties that $\sum_{j=1}^{J} \btheta_{j} =0$ and $\sum_{j=1}^{J}(\btheta_{j}-\btheta_{j}^*) = -\sum_{j=1}^{J}\btheta_{j}^* = J \bar\btheta^*$, it follows from elementary algebra that
$
\frac{1}{J}\sum_{j=1}^{J-1}\sum_{j'=j+1}^{J}\left| (\btheta_{j}-\btheta_{j}^*) -(\btheta_{j'}-\btheta_{j'}^*) \right|^2 
=  \sum_{j=1}^{J}(\btheta_{j}- (\btheta^*_j - \bar\btheta^*))^2.
$
This equality together with the lower bound \eqref{ed:minD} completes the proof. \hfill$\Box$

\section{Proof of the results in Section \ref{part:uppershift}}

\subsection{Proof of Theorem \ref{theo:shift1}} \label{part:ps}

Let us state the following lemma which is direct consequence of Bernstein's inequality for bounded random variables (see e.g.\ Proposition 2.9 in \cite{MR2319879}):
\begin{lemme} \label{lem:Bern}
Suppose that Assumption \ref{hyp:g} holds. Then, for any $x > 0$
$$
\P \bigg(  \frac{1}{J} \lVert \btheta^*_{\bTheta_{0}} - \btheta^* \rVert^2 \geq \rho^2 \bigg( \sqrt{\frac{2 x}{J}} + \frac{x}{3J} \bigg)^2  \bigg) \leq 2 e^{-x}.
$$
\end{lemme}

Using the inequality $\frac{1}{J} \lVert \hat\btheta_{\lambda} - \btheta^* \rVert^2 \leq \frac{2}{J} \lVert \hat\btheta_{\lambda} - \btheta^*_{\bTheta_{0}} \rVert^2 + \frac{2}{J} \lVert \btheta^*_{\bTheta_{0}} - \btheta^* \rVert^2$, it follows that Theorem \ref{theo:shift1} is a consequence of Lemma \ref{lem:Bern} and Theorem \ref{th:shift}. Indeed, it can be easily checked  that, under the assumptions of  Theorem \ref{theo:shift1},  Assumptions \ref{ass:op12} to \ref{hyp:EigZmax} are satisfied in the case of randomly shifted curves with an equi-spaced design and low-pass Fourier filtering, see the various arguments given in Section \ref{part:ab}). The identifiability condition of Assumption \ref{hyp:ident} is given by Proposition \ref{prop:D}. \hfill$\Box$

\subsection{Proof of Theorem \ref{theo:shift2}}

Consider the following inequality 
$
\lVert \hat f^{\lambda} - f \rVert^2 \leq 2 \lVert \hat f^{\lambda} - f_{\bTheta_{0}} \rVert^2 + 2 \lVert f_{\bTheta_{0}} - f \rVert^2,
$
where $f_{\bTheta_0} (t)= f(t-\bar \btheta^*)$ and $\bar \btheta^* = \frac{1}{J}\sum_{j=1}^J\btheta_j^* \in\Theta$. As $f$ is assumed to be in $H_s(A)$, there  exists a constant $ C(\Theta,f) > 0$ such that $ \lVert f_{\bTheta_{0}} - f \rVert^2_{L^2}  \leq C(\Theta,f)| \bar \btheta^*|^2 = C(\Theta,f) \frac{1}{J}\| \btheta^*_{\bTheta_0} - \btheta^*\|^2$. As explained in part \ref{part:ps} the assumptions of Theorem \ref{theo:shift2} are   satisfied in the case of randomly shifted curves with an equi-spaced design and low-pass Fourier filtering. The result then follows from  Theorem \ref{th:funct}. \hfill $\Box$

\subsection{Proof of Theorem \ref{th:vantreefunct}} \label{part:proofvtfunct}

Let $n \geq 1$. We have that
\begin{align}
\E[\snorm{\tilde f - f }_{L^2}]  = \E \snorm{\tilde f -  f_{\bTheta_0}+  f_{\bTheta_0} -  f }_{L^2} 
 \geq \Big| \; \underbrace{\E \snorm{\tilde f - f_{\bTheta_0}}_{L^2}}_{\mathbf A} -  \underbrace{\E \snorm{ f_{\bTheta_0} - f}_{L^2}  }_{\mathbf B}\; \Big | \label{eq:Min1}
\end{align}
where for all $t\in [0,1]$, 
$
\tilde f  (t) = \frac{1}{J} \sum_{j=1}^J f( t-\btheta^*_j +\hat\btheta^{\lambda}_j),
$ and
$
 f_{\bTheta_0}  (t)  =  f(t + \bar \btheta^*)
$, 
with $\bar \btheta^{*} = \tfrac{1}{J} \sum_{j=1}^J \btheta^{*}_j$. In the rest of the proof, we show that $\mathbf A$ is bounded from below by a quantity $C_{0}(f,g,n,\sigma^2,\rho) = C(f,\rho)  \frac{n^{-1}\sigma^2 }{ \norm{\partial_{t} f}_{\infty}^2  + n^{-1}\sigma^2 \int_{\Theta}\left(\partial_{\btheta } \log\left(g(\btheta )\right)\right)^2 }$ independent of $J$ (this statement is made precise later) and that $\mathbf B$ goes to zero as $J$ goes to infinity. Then, these two facts imply that there exists a $J_0 \in \N$ such that $J\geq J_0$ implies that
$
\E \snorm{\tilde f - \tilde f }_{L^2} \geq\tfrac{1}{2} C_{0}(f,g,n,\sigma^2,\rho),
$
which will yield the desired result.

\paragraph*{Lower bound on $\mathbf A$.} Recall that $c_{k}^{\ast} = \int_0^1 f(t) e^{-i2\pi k t}dt $, then
\begin{align*}
\snorm{\tilde f - f_{\bTheta_0}}_{L^2} & = \snorm{ \frac{1}{J} \sum_{j=1}^J f( \cdot-\btheta^*_j +\hat\btheta_j^{\lambda})  -  f(\cdot + \bar\btheta^*)}_{L^2} =  \bigg(\sum_{k\in\Z} \bigg |\frac{1}{J} \sum_{j=1}^J \Big( e^{i2\pi k (-\btheta^*_j +\hat\btheta_j^{\lambda} )} - e^{i2\pi k \bar\btheta^*} \Big)  c_{k}^{\ast}  \bigg |^2 \bigg)^{\tfrac{1}{2}}, \\
& \geq  |c^*_{1} |\bigg |\frac{1}{J} \sum_{j=1}^J( e^{i2\pi  (\hat\btheta^{\lambda}_{j}- [\btheta^{*}_{\bTheta_0}]_{j})} -1 )\bigg |,
\end{align*}
where $\btheta^{*}_{\bTheta_0} = (\btheta^*_1 - \bar\btheta^* ,\ldots,\btheta^*_J- \bar\btheta^* ) $, the right hand side of the preceding inequality being positive since Assumption \ref{hyp:g} ensures that $c^*_{1}\neq 0 $ for all $j=1,\ldots,J$. Let $u_{j} = 2 \pi( \hat\btheta^{\lambda}_{j}- [\btheta^{*}_{\bTheta_0}]_{j}), j=1,\ldots,J$ . Since $\sum_{j=1}^J u_{j} =0$ and $|u_{j}| \leq 4 \pi \rho < 3, \ j=1,\ldots,J$ (by our assumption on $\rho$), Lemma \ref{lemme:SumExp} implies that 
\begin{equation}\label{eq:Min0}
\snorm{\tilde f -  f_{\bTheta_0}}_{L^2}\geq C(f,\rho) \frac{1}{J}\snorm{\hat\btheta^{\lambda}- \btheta^{*}_{\bTheta_0}}^2 .
\end{equation}
Now, remark that $\E\big[\tfrac{1}{J}\snorm{\hat\btheta^{\lambda}- \btheta^{*}_{\bTheta_0}}^2 \big] \geq  \E\big[\tfrac{1}{J}\snorm{\hat\btheta^{\lambda}- \btheta^{*}}^2 \big]  - \mathbf{C}$ with $\mathbf{C} = 2 \E \big[ \abs{\bar\btheta^*}  \frac{1}{J} \sum_{j=1}^{J}  |\hat\btheta^{\lambda}_{j} -\btheta^*_j | \big]$.
By applying Theorem \ref{theo:vantreeshift} we get that
$$
\E\big[\tfrac{1}{J}\snorm{\hat\btheta^{\lambda}- \btheta^{*}}^2 \big] \geq  C(f,g,n,\sigma^2), \mbox{ with } C(f,g,n,\sigma^2) =  \frac{n^{-1}\sigma^2 }{ \norm{\partial_{t} f}_{\infty}^2  + n^{-1}\sigma^2 \int_{\Theta}\left(\partial_{\btheta } \log\left(g(\btheta )\right)\right)^2 }. 
$$ 
Then, remark that $\mathbf C \leq 4 \rho \sqrt{ \E  \abs{\bar\btheta^*}^2 } \leq C(\rho,g) J^{-1/2}$.
Hence $\mathbf C$ tends to 0 as $J$ goes to infinity. Therefore, using equation \eqref{eq:Min0}, it follows that there exists $C_{0}(f,g,n,\sigma^2,\gamma,\rho) > 0$ and $J_1\in\N$ such that $J\geq J_1$ implies that
\begin{equation}
\mathbf A = \E\big[\snorm{\tilde f^\lambda - \tilde f}_{L^2} \big] \geq C_{0}(f,g,n,\sigma^2,\rho).\label{eq:Min3}
\end{equation}

\paragraph*{Upper bound on $\mathbf B$.} By assumption, $f$ is continuously differentiable on $[0,1]$ implying that
$
\snorm{f_{\bTheta_0} -  f}_{L^2}  = \snorm{ f( \cdot+ \bar \btheta^*  ) - f }_{L^2} \leq \norm{\partial_t f}_{\infty} \sabs{\bar \btheta^{*} }.
$
Therefore, $\E \snorm{f_{\bTheta_0} -  f}_{L^2} \leq \norm{\partial_t f}_{\infty} \sqrt{ \E  \abs{\bar\btheta^*}^2 } \leq   C(f,g) J^{-1/2} $. Hence,  there exists a $J_2\in\N$ such that $J\geq J_2$ implies  
\begin{equation}
\mathbf B = \E[\snorm{\tilde f_{\bTheta_0} - \tilde f}_{L^2}] \leq \frac{1}{2}C_{0}(f,g,n,\sigma^2,\rho). \label{eq:Min2}
\end{equation}
To conclude the proof, equations \eqref{eq:Min1}, \eqref{eq:Min3} and \eqref{eq:Min2} imply that there exists a $J_0 \in \N$ such that  $J\geq J_0$  implies
$
\E  \snorm{\hat f^\lambda - \tilde f }_{L^2}  \geq | \mathbf A - \mathbf B |  \geq  \frac{1}{2}C_{0}(f,g,n,\sigma^2,\rho).
$ \hfill $\Box$

\section{Proof of the results in Section \ref{sec:gen}}

\subsection{Proof of Theorem \ref{th:shift}}\label{part:proofSIM}

We explain here the main arguments of the proof of Theorem \ref{th:shift}. Technical Lemmas are given in the second part of the Appendix. Let $\btheta= (\btheta_1,\ldots, \btheta_J) = (\theta_{1}^{1},\ldots,\theta_{1}^{p},\ldots,\theta_{J}^{1},\ldots,\theta_{J}^{p}) \in\R^{pJ}$ and decompose the criterion \eqref{eq:critM} as follows,
\begin{align*}
M_{\lambda}(\btheta)&= \frac{1}{J}\sum_{j=1}^{J} \int_\Omega \bigg ( \tilde T_{\btheta_{j}}\prs{ S_{\lambda_n}(t),\Y_{j}} - \frac{1}{J} \sum_{j'=1}^{J} \tilde T_{\btheta_{j'}}\prs{S_{\lambda_n}(t),\Y_{j'}}\bigg)^2dt\\
&=D(\btheta) + \Big[ R_{\lambda}(\btheta) + Q_{\lambda}(\btheta) \Big] + \Big[ Q^Z_{\lambda}(\btheta) + R^{Z}_{\lambda}(\btheta) + R^{Z,\varepsilon}_{\lambda}(\btheta)+ Q^{\varepsilon}_{\lambda}(\btheta) +  R^{\varepsilon}_{\lambda}(\btheta) \Big],
\end{align*}
where 
$
D(\btheta) = \frac{1}{J}\sum_{j=1}^{J} \int_\Omega \bigg(\tilde T_{\btheta_{j}} T_{\btheta_{j}^*}f(t) - \frac{1}{J} \sum_{j'=1}^{J} \tilde T_{\btheta_{j'}} T_{\btheta_{j'}^*}f(t) \bigg)^2 dt,
$
the terms $R_\lambda$ and $Q_\lambda$ are due to the smoothing, namely,
\begin{align*}
Q_{\lambda}(\btheta)& = \frac{1}{J} \sum_{j=1}^{J} \int_\Omega \bigg( \tilde T_{\btheta_{j}}B_\lambda(T_{\btheta_{j}^*}f,t) - \frac{1}{J} \sum_{j'=1}^{J}\tilde T_{\btheta_{j'}} B_\lambda(T_{\btheta_{j'}^*}f,t)\bigg)^2dt \\
\begin{split}
R_{\lambda}(\btheta) & = \frac{1}{J}\sum_{j=1}^{J} \int_\Omega \bigg(\tilde T_{\btheta_{j}} T_{\btheta_{j}^*}f(t)- \frac{1}{J} \sum_{j'=1}^{J}\tilde T_{\btheta_{j'}} T_{\btheta_{j'}^*}f(t)\bigg)\\
&\hspace{5cm}\times\bigg( \tilde T_{\btheta_{j}}B_\lambda(T_{\btheta_{j}^*}f,t) - \frac{1}{J} \sum_{j'=1}^{J}\tilde T_{\btheta_{j'}} B_\lambda(T_{\btheta_{j'}^*}f,t)  \bigg )dt,
\end{split}
\end{align*}
and the others terms contain the $Z_j$'s and $\beps_j$'s error terms. Let $\T_{\btheta^*_j}\bZ_j = \big(T_{\btheta^*_j}Z_j(t_\ell)\big)_{\ell=1}^n $ and $ \T_{\btheta^*_j}\f = \big(T_{\btheta^*_j}f(t_\ell)\big)_{\ell=1}^n $, then
\begin{align*}
Q^{Z}_{\lambda}(\btheta) & = \frac{1}{J} \sum_{j=1}^{J} \int_\Omega\bigg( \tilde T_{\btheta_{j}}\prs{S_\lambda(t),\T_{\btheta^*_j}\bZ_j} - \frac{1}{J} \sum_{j'=1}^J \tilde T_{\btheta_{j'}}\prs{S_\lambda(t), \T_{\btheta^*_{j'}}\bZ_{j'}}\bigg)^2dt \\
R^{Z}_{\lambda}(\btheta)& = \frac{2}{J} \sum_{j=1}^{J} \int_\Omega \bigg(\tilde T_{\btheta_{j}}\prs{ S_\lambda(t),\T_{\btheta_j^*} \f} - \frac{1}{J} \sum_{j'=1}^J \tilde T_{\btheta_{j'}}\prs{S_\lambda(t), \T_{\btheta_{j'}^*} \f} \bigg)\\&\hspace{5cm} \times \bigg( \tilde T_{\btheta_{j}}\prs{S_\lambda(t),\T_{\btheta_{j}^*}\bZ_j} - \frac{1}{J} \sum_{j'=1}^J \tilde T_{\btheta_{j'}}\prs{S_\lambda(t),\T_{\btheta_{j'}^*}\bZ_{j'}} \bigg)dt,\allowdisplaybreaks \\
R^{Z,\varepsilon}_{\lambda}(\btheta) & = \frac{2\sigma}{J} \sum_{j=1}^{J} \int_\Omega\bigg( \tilde   T_{\btheta_{j}}\prs{S_\lambda(t),\T_{\btheta_{j}^*}\bZ_j} - \frac{1}{J} \sum_{j'=1}^J \tilde T_{\btheta_{j'}}\prs{S_\lambda(t),\T_{\btheta_{j'}^*}\bZ_{j'} } \bigg) \\ &\hspace{5cm} \times \bigg(  \tilde T_{\btheta_{j}}\prs{S_\lambda(t),\bvarepsilon_j} - \frac{1}{J} \sum_{j'=1}^J \tilde T_{\btheta_{j'}}\prs{S_\lambda(t), \bvarepsilon_{j'}} \bigg) dt\allowdisplaybreaks\\
Q^{\varepsilon}_{\lambda}(\btheta) & = \frac{\sigma^2}{J} \sum_{j=1}^{J} \int_\Omega\bigg( \tilde T_{\btheta_{j}}\prs{S_\lambda(t),\bvarepsilon_j} - \frac{1}{J} \sum_{j'=1}^J \tilde T_{\btheta_{j'}}\prs{S_\lambda(t), \bvarepsilon_{j'}}\bigg)^2dt\\
R^{\varepsilon}_{\lambda}(\btheta) & = \frac{2\sigma}{J} \sum_{j=1}^{J}  \int_\Omega \bigg( \tilde T_{\btheta_{j}}\prs{ S_\lambda(t),\T_{\btheta_j^*} \f} - \frac{1}{J} \sum_{j'=1}^J  \tilde T_{\btheta_{j'}}\prs{S_\lambda(t),\T_{\btheta_{j'}^*} \f } \bigg)\\& \hspace{5cm}\times\bigg(  \tilde T_{\btheta_{j}}\prs{S_\lambda(t),\bvarepsilon_j} - \frac{1}{J} \sum_{j'=1}^J \tilde T_{\btheta_{j'}}\prs{S_\lambda(t), \bvarepsilon_{j'}} \bigg)dt.
\end{align*}
At this stage, recall that $
\btheta_{\bTheta}^*  = \argmin_{\btheta \in \bTheta} D(\btheta) 
$
and
$
\hat\btheta^\lambda  = \argmin_{\btheta \in \bTheta} M_\lambda(\btheta).
$
The proof follows a classical guideline in M-estimation: we show that the uniform (over $\bTheta$) convergence in probability of the criterion $M_\lambda$ to $D$, yielding the convergence in probability of their argmins $\btheta_{\bTheta}^*$  and $\hat\btheta^\lambda$ respectively. 
Assumption \ref{hyp:ident} ensures that there is a constant $C(\Theta,\F,f)>0$ such that,
\begin{equation}\label{eq:1}
\frac{1}{J}\lVert\hat\btheta^\lambda - \btheta_{\bTheta}^*\rVert^2 \leq C(\Theta,\bTheta,\F,f) \abs{D(\hat\btheta^\lambda) - D(\btheta^*_{\bTheta})} 
\end{equation}
Then, a classical inequality in M-estimation and the decomposition of $M_\lambda(\btheta)$  given above yield
\begin{align}\label{eq:2}
\abs{D(\hat\btheta^\lambda) - D(\btheta^*_{\bTheta}) } & \leq 2\sup_{\btheta\in\bTheta}\abs{D(\btheta) - M_\lambda(\btheta)} \\
& \quad  = \underbrace{2\sup_{\btheta\in\bTheta}\Big\{ R_{\lambda}(\btheta) + Q_{\lambda}(\btheta) \Big\}}_{\mathbf B} +\underbrace{2\sup_{\btheta\in\bTheta} \Big\{ Q^Z_{\lambda}(\btheta) + R^{Z}_{\lambda}(\btheta) + R^{Z,\varepsilon}_{\lambda}(\btheta)+ Q^{\varepsilon}_{\lambda}(\btheta) +  R^{\varepsilon}_{\lambda}(\btheta) \Big\}}_{\mathbf V} \nonumber
\end{align}
The rest of the proof is devoted to control the $\mathbf B$ and $\mathbf V$ terms. 
\paragraph*{Control of $\mathbf B$.}  Using Assumption \ref{ass:smoo} and \ref{ass:op12}, we have that
$
Q_\lambda(\btheta)
\leq \frac{C(\Theta)}{J}\sum_{j=1}^J \norm{B_\lambda(T_{\btheta^*_j}f,t ) }^2_{L^2} \leq C(\Theta,\F) B(\lambda).
$
Now by applying the Cauchy-Schwarz inequality, $\abs{R_{\lambda}(\btheta)} \leq \sup_{\btheta\in\bTheta}\limits \{\sqrt{D(\btheta)}\} \sqrt{Q_{\lambda}(\btheta)} $. By Assumption \ref{ass:op12}, there exists a constant such $ \sup_{\btheta\in\bTheta}\limits \{ {D(\btheta)}\} \leq C(\Theta,\F,f)$ and thus
\begin{equation}\label{eq:B}
\mathbf B 
\leq C(\Theta,\F,f) \big(B(\lambda)+ \sqrt{B(\lambda})\big).
\end{equation}

\paragraph*{Control of $\mathbf V$.}  We give a control in probability of the stochastic quadratic term $Q_\lambda^Z$ and $Q_\lambda^\varepsilon$.  As previously, one can show that there is a constant $C(\Theta,\F,f)>0$ such that,
\begin{align*}
\abs{Q^Z_{\lambda}(\btheta) + R^{Z}_{\lambda}(\btheta) + R^{Z,\varepsilon}_{\lambda}(\btheta)+ Q^{\varepsilon}_{\lambda}(\btheta) +  R^{\varepsilon}_{\lambda}(\btheta)} \leq  C(\Theta,\F,f)\left(\sqrt{Q^Z_\lambda(\btheta)}+ Q^Z_\lambda(\btheta) + Q^\varepsilon_\lambda(\btheta) +  \sqrt{Q^\varepsilon_\lambda(\btheta)}  \right),
\end{align*}
where we have used the inequality $2ab \leq a^2 +b^2$, valid for any $a,b>0$ to control the term $ R^{Z,\varepsilon}_{\lambda}$. 
The quadratic terms $Q_\lambda^Z$ and $Q_\lambda^\varepsilon$ are controlled by Corollaries \ref{cor:Qepscon} and \ref{cor:QZcon} respectively. It yields immediately to
\begin{equation}\label{eq:V}
\P\left( \mathbf V \geq   C(\Theta,\F,f) (\gamma_{\max}(n) + \sigma^2) \big(\upsilon(x,J,\lambda) + \sqrt{\upsilon(x,J,\lambda)} \big) \right) \leq 2 e^{-x},
\end{equation}
where $\upsilon(x,J,\lambda) = V(\lambda) \left(1 + 4 \frac{x}{J} +\sqrt{4\frac{x}{J}} \right)$.

Putting together equations \eqref{eq:1}, \eqref{eq:2},  \eqref{eq:B} and \eqref{eq:V}, we have 
\begin{align*}
\P\bigg(\frac{1}{J} \lVert\btheta_{\bTheta}^* -\hat\btheta^{\lambda} \rVert^2 \geq  C(\Theta,\bTheta,\F,f) \Big[ (\gamma_{\max}(n)+\sigma^2)\Big( \sqrt{\upsilon(x,J,\lambda)} + \upsilon(x,J,\lambda)\Big) + \Big(B(\lambda) + \sqrt{B(\lambda)}\Big)\Big]& \bigg) \leq 2e^{-x},
\end{align*}
which completes the proof of Theorem \ref{th:shift}. \hfill$\Box$

\subsection{Proof of Theorem \ref{th:funct}}

In this part, we use the notations introduced in the proof of Theorem \ref{th:shift}. We have,
\begin{align*}
\norm{ f_{\bTheta}^* -\hat f^\lambda}^2_{L^2}& \leq \underbrace{ \frac{2}{J} \sum_{j=1}^J\norm{ \tilde T_{[\btheta_{\bTheta}^*]_j }T_{ \btheta_{j}^{*}}f -\tilde T_{[\btheta_{\bTheta}^*]_j}\prs{S_\lambda(\cdot),  \T_{\btheta_j^*} \f} }^2_{L^2}}_{\mathbf{B'} } \\&  \qquad+  \underbrace{ \frac{2}{J} \sum_{j=1}^J\norm{  \tilde T_{ [\btheta_{\bTheta}^*]_j}\prs{S_\lambda(\cdot), \T_{\btheta_j^*} \f}-\tilde T_{ \hat\btheta^{\lambda}_{j}}\prs{S_\lambda(\cdot),\Y_j}}^2_{L^2}}_{\mathbf{V'} } .
\end{align*}
Again, the first term above depends on the bias, and the second term (stochastic)  can be controlled in probability. 
Under Assumptions \ref{ass:op12} and \ref{ass:smoo} we have that
\begin{align*}
\mathbf{B'}  
& \leq \frac{C(\Theta)}{J} \sum_{j=1}^J \norm{\prs{S_\lambda(\cdot),\T_{ \btheta^*_j}\f }-  T_{\btheta^*_j}f}^2_{L^2}  \leq C(\Theta,\F) B(\lambda),
\end{align*}
and
\begin{align*}
\mathbf{V'} & = \frac{2}{J} \sum_{j=1}^J\norm{\tilde T_{[\btheta_{\bTheta}^*]_j}\prs{S_\lambda(\cdot), \T_{\btheta_j^*} \f} -  \tilde T_{ \hat\btheta^{\lambda}_{j}}\prs{S_\lambda(\cdot),\T_{\btheta_j^*} \f} +\tilde T_{ \hat\btheta^{\lambda}_{j}}\prs{S_\lambda(\cdot),\T_{\btheta_j^*} \f}- \tilde T_{ \hat\btheta^{\lambda}_{j}}\prs{S_\lambda(\cdot),\Y_j} }^2_{L^2} \\
& \leq \frac{C(\Theta,\F)}{J} \sum_{j=1}^J \left(\|\hat\btheta^{\lambda}_{j} - [\btheta_{\bTheta}^*]_j \|^2  +  \norm{\prs{S_\lambda(\cdot),\Y_j-\T_{\btheta_j^*} \f}}^2_{L^2} \right),\allowdisplaybreaks \\
& \leq  C(\Theta,\F)\bigg( \frac{1}{J}\|\hat\btheta^{\lambda} - \btheta_{\bTheta}^*\|^2+ \frac{1}{J}  \sum_{j=1}^J  \norm{\prs{S_\lambda(\cdot), \T_{\btheta^*_j}\bZ_j + \bvarepsilon_{j}   }}^2_{L^2} \bigg)
\end{align*}
The stochastic  term $ \frac{1}{J}  \sum_{j=1}^J  \norm{\prs{S_\lambda(\cdot), \T_{\btheta^*_j}\bZ_j + \bvarepsilon_{j}   }}^2_{L^2} $ in the above inequality can be been controlled using Lemma \ref{lemme:conNoise} and  the arguments in the proof of Corollaries \ref{cor:Qepscon} and \ref{cor:QZcon} to obtain that for any $x > 0$
$$
\P\bigg( \frac{1}{J}  \sum_{j=1}^J  \norm{\prs{S_\lambda(\cdot), \T_{\btheta^*_j}\bZ_j + \bvarepsilon_{j}   }}^2_{L^2}  \geq  C(\Theta,\F,f)  (\gamma_{\max}(n)+\sigma^2)\Big( \sqrt{\upsilon(x,J,\lambda)} + \upsilon(x,J,\lambda)\Big) \bigg) \leq e^{-x}.
$$
Then, from Theorem \ref{th:shift} it follows that 
\begin{align*}
\P\bigg(\mathbf{B'} + \mathbf{V'} \geq  C(\Theta,\bTheta,\F,f) \Big[ (\gamma_{\max}(n)+\sigma^2)\Big( \sqrt{\upsilon(x,J,\lambda)} + \upsilon(x,J,\lambda)\Big) + \Big(B(\lambda) + \sqrt{B(\lambda)}\Big)\Big]& \bigg) \leq 2e^{-x},
\end{align*}
which completes the proof. \hfill $\Box$

\section{Technical Lemmas}

\begin{lemme} \label{lemme:SumExp}
Let $u=(u_1,\ldots,u_J)$ such that $\sum_{j=1}^J  u_j = 0$ with $| u_j |\leq \delta $ for some $0 \leq \delta < 3$ for all $j=1,\ldots,J$. Then, there exists a constant $C(\delta) > 0$ such that
$
\big|\frac{1}{J} \sum_{j=1}^J ( e^{i u_j} -1 )\big| \geq \frac{C(\delta)}{J} \norm{u}^2
$
where $\norm{u}^2 = u_1^2 + \ldots + u_J^2$.
\end{lemme}

\begin{proof}
Let $F(u_1,\ldots,u_J) = \frac{1}{J} \sum_{j=1}^J e^{iu_j}$. A  Taylor expansion implies that there exits  $t_j \in [- \delta, \delta]$, $j=1,\ldots,J$ such that 
$$
F(u_1,\ldots,u_J) = 1 + \frac{i}{J} \sum_{j=1}^J u_j - \frac{1}{2J} \sum_{j=1}^J u_j^2  - \frac{i}{6J} \sum_{j=1}^J u_j^3 e^{it_j},
$$
holds for all $| u_j |\leq \delta$. Now, since $\sum_{j=1}^J u_j = 0$ it follows that
\begin{align*}
\babs{ \frac{1}{J} \sum_{j=1}^J e^{iu_j} - 1} & = \babs{- \frac{1}{2J} \sum_{j=1}^J u_j^2  - \frac{i}{6J} \sum_{j=1}^J u_j^3 e^{it_j} } \geq  \frac{1}{2J}  \babs{ \sum_{j=1}^J u_j^2  - \bigg |   \frac{i}{3} \sum_{j=1}^J u_j^3 e^{it_j}   \bigg |\; }. 
\end{align*}
Since $| u_j | \leq \delta$, we have that
$
\abs{   \frac{i}{3} \sum_{j=1}^J u_j^3 e^{it_j}} \leq \frac{\delta}{3} \sum_{j=1}^J \abs{u_j}^2 
$
which finally implies that
$
\abs{ \frac{1}{J} \sum_{j=1}^J e^{iu_j} - 1} \geq \frac{3 - \delta}{6} \frac{1}{J} \sum_{j=1}^J u_j^2 ,
$
which proves the result by letting $C(\delta)= \frac{3 - \delta}{6} > 0$ since $\delta < 3$.
\end{proof}

\begin{lemme}\label{lemme:conNoise}
Let $\xi_{\lambda,J}(A_1,\ldots,A_J) = \frac{1}{J}\sum_{j=1}^{J} \limits\norm{\prs{S_{\lambda}(\cdot),A_j\bvarepsilon_j}}^2_{L^2}$, where $\bvarepsilon_j \sim \mathcal{N}(0,I_{n})$ and the $A_j$'s are nonrandom non-negative $n\times n$ symmetric matrices. Then, for all $ x> 0$ and all $n\geq 1$,
$$
\P\left( \xi_{\lambda,J}(A_1,\ldots,A_J)\geq \frac{1}{J} \norm{\mathbf A} \Big( 1 + 4 \frac{x}{J} +\sqrt{ 4\frac{x}{J} } \Big) \right) \leq e^{-x}.
$$
where $\norm{\mathbf A} = \sum_{j=1}^J\limits \sum_{\ell=1}^n\limits r_{j,\ell}$ with $ r_{j,\ell}$ being the $\ell$-th eigenvalue of the matrix $\mathbf A_j = A_j \left[\sprs{ S^{\ell}_{\lambda}, S^{\ell'}_{\lambda}}_{L^2}\right]_{\ell,\ell'=1}^n A_j$. 
\end{lemme}
\begin{proof}
Some parts of the proof follows the arguments in \cite{MR1653272} (Lemma 8, part 7.6). We have
\begin{align*}
\xi_{\lambda,J} = \frac{1}{J}\sum_{j=1}^{J}\bigg\|\sum_{\ell=1}^n  S^\ell_{\lambda}(\cdot)  [A \bvarepsilon_j]^{\ell} \bigg\|_{L_2}^2 = \frac{1}{J}\sum_{j=1}^{J} \sum_{\ell,\ell'=1}^n \sprs{S^{\ell}_{\lambda}, S^{\ell'}_{\lambda}}_{L^2} [A_j \bvarepsilon_j]^{\ell}[A_j \bvarepsilon_j]^{\ell'} = \frac{1}{J}\sum_{j=1}^{J}\bvarepsilon_j' \mathbf A_j \bvarepsilon_j,
\end{align*}
where $\mathbf A_j = A_j \mathbf{S}_\lambda A_j$ with $\mathbf S_{\lambda} = \left[\sprs{ S^{\ell}_{\lambda}, S^{\ell'}_{\lambda}}_{L^2}\right]_{\ell,\ell'=1}^n$. Now, denote by $r_{j,1} \geq\ldots\geq r_{j,n}$ the eigenvalues of $\mathbf A_j$ with $ r_{j,1}\geq \ldots\geq r_{j,n}\geq 0$ and $r_1 = \max_{j, \ell}\{r_{j,\ell}\}$. We can write $\mathbf A_j = ({\mathbf S_{\lambda}}^{\frac{1}{2}} A_j)' ({\mathbf S_{\lambda}}^{\frac{1}{2}}A_j)$ and is positive semi-definite. Then, let 
$
\tilde \xi_{\lambda,J} = J\xi_{\lambda,J} - J\Ec \xi_{\lambda,J} = \sum_{j=1}^J\limits (\beps_j' \mathbf A_j \beps_j - \tr \mathbf{A}_j).
$
Let $\alpha>0$, by Markov's inequality it follows that for all $u \in \left (0,\frac{1}{2r_1} \right)$, 
$
\P\left(\tilde \xi_{\lambda,J}\geq \alpha\right) = \P\left(e^{u\tilde \xi_{\lambda,J}} \geq e^{u\alpha }\right)  \leq e^{-u\alpha }\prod_{j=1}^{J}\Ec\left[e^{u {\beps_j}' \mathbf{A}_j \beps_j -u\tr \mathbf{A}_j }\right],
$
since the $\beps_j$'s are independent. The log-Laplace transform of $\tilde\varphi_{\lambda,j} = {\beps_j}' \mathbf{A}_j \beps_j - \tr \mathbf{A}_j $ is
$
\log \left(\E\left [ e^{u\tilde\varphi_{\lambda,j}} \right]\right ) =\sum_{\ell=1}^n -u r_{j,\ell} -\frac{1}{2} \log\left( 1- 2u r_{j,\ell}\right).
$
We now use the inequality $-x - \frac{1}{2}\log(1-2x) \leq \frac{x^2}{1-2x}$ for all $0<x<\frac{1}{2}$ which holds since $u \in \left (0,\frac{1}{2r_1} \right)$. This implies that
$
\log \left(\E\left [ e^{u\tilde\varphi_{\lambda,j}} \right]\right ) \leq \sum_{\ell=1}^n\frac{u^2 {r_{j,\ell}}^2}{1-2u r_{j,\ell}} \leq \frac{u^2 \norm{r_j}^2}{1-2u r_1},
$
where $\norm{r_j}^2 = r_{j,1}^2+\ldots+r_{n,j}^2$.  Finally, we have
\begin{equation} \label{eq:u}
\P\left(\tilde \varphi_{\lambda,J}\geq \alpha\right) \leq \exp\bigg( -\bigg( u\alpha - \sum_{j=1}^J\frac{\norm{r_j}^2 u^2}{1-2 r_{1}u}\bigg) \bigg) =\exp\bigg( -\bigg( u\alpha - \frac{\norm{r}^2 u^2}{1-2 r_{1}u}\bigg) \bigg) ,
\end{equation}
where $\norm{r}^2 = \sum_{j=1}^J \sum_{\ell=1}^n {r_{j,\ell}}^2$. The right hand side of the above inequality achieves its minimum at $u = \frac{1}{2r_1}\left(1 - \frac{\norm{r}}{\sqrt{2\alpha r_1+\norm{r}^2} } \right) $. Evaluating (\ref{eq:u}) at this point and using the inequality $ (1+x)^{1/2} \leq 1 + \frac{x}{2}$, valid for all $x\geq -1$, one has that
\begin{align*}
\P\left(\tilde \xi_{\lambda,J}\geq \alpha\right)  & \leq \exp\left(-\frac{\alpha^2}{2r_1\alpha + 2\norm{r}^2 +2\norm{r}^2(1+4\alpha r_1 / (2\norm{r}^2) )^{1/2}} \right)\\ &\leq \exp\left(-\frac{\alpha^2}{4r_1\alpha + 4\norm{r}^2} \right),
\end{align*}
by setting $x=\frac{\alpha^2}{4r_1\alpha + 4\norm{r}^2}$. We derive the following concentration inequality for $\xi_{\lambda,J}= \frac{1}{J} \tilde \xi_{\lambda,J} + \frac{1}{J}\sum_{j=1}^J\tr(\mathbf A_j)$,
$
\P\bigg( \xi_{\lambda,J} \geq \frac{1}{J}\sum_{j=1}^J\sum_{\ell=1}^n r_{j,\ell} + 4 \frac{r_1}{J}x  + \frac{\norm{r}}{J}\sqrt{4x  } \bigg) \leq e^{-x}.
$
To finish the proof, remark that $\norm{r}^2 = \sum_{j=1}^J \sum_{\ell=1} r_{j,\ell}^2 \leq \left( \sum_{j=1}^J \sum_{\ell=1}^n r_{j,\ell}\right)^2$ since all the $r_{j,\ell}$'s are positive.
\end{proof}

\begin{corollaire}\label{cor:Qepscon}
Under Assumptions \ref{ass:op12} to \ref{ass:smoo}, there exists a constant $C(\Theta,\F)  > 0$ such that for all $x > 0$,
$$
\P\left(\sup_{\btheta\in\bTheta} Q^{\varepsilon}_{\lambda}(\btheta) \geq C(\Theta,\F) \sigma^2 V(\lambda)\left(1 + 4 \frac{x}{J} +\sqrt{4\frac{x}{J}} \right)\right) \leq e^{-x}.
$$
\end{corollaire}
\begin{proof}
Assumption \ref{ass:op12} gives the uniform bound 
\begin{align*}
Q^{\varepsilon}_{\lambda}(\btheta)  
& \leq \frac{1}{J}\sum_{j=1}^{J} \int_\Omega\left( \tilde T_{\btheta_{j}}\prs{S_{\lambda}(t),\sigma\bvarepsilon_j} \right)^2dt \leq \frac{C(\Theta,\F)}{J} \sum_{j=1}^{J} \norm{\prs{S_{\lambda}(t),\sigma\bvarepsilon_j}}_{L^2}^2 \\& = C(\Theta,\F)\xi_{\lambda,J}(\sigma Id_n,\ldots,\sigma Id_n),
\end{align*}
where $\xi_{\lambda,J}(\sigma Id_n,\ldots,\sigma Id_n)$ is defined in Lemma \ref{lemme:conNoise} and does not depend on $\btheta$. Thus, the result immediately follows from Lemma \ref{lemme:conNoise}.
\end{proof}

\begin{corollaire}\label{cor:QZcon}
Under Assumptions \ref{ass:op12} to \ref{hyp:EigZmax}, there exists a constant $C(\Theta,\F)  > 0$ such that  for all $x\geq 0$,
$$
\P\left(\sup_{\btheta\in\bTheta} Q^Z_{\lambda}(\btheta) \geq  C(\Theta,\F)\gamma_{n}(\Theta)  V(\lambda)\left(1 + 4 \frac{x}{J} +\sqrt{4\frac{x}{J}} \right)\right) \leq e^{-x}.
$$
\end{corollaire}
\begin{proof}
Assumption \ref{ass:op12} gives the uniform bound 
\begin{align*}
Q^{Z}_{\lambda}(\btheta) 
& \leq \frac{1}{J}\sum_{j=1}^{J} \int_\Omega\left( \tilde T_{\btheta_{j}}\prs{S_{\lambda}(t),\T_{\btheta_{j}^*} \bZ_{j}} \right)^2dt \leq \frac{C(\Theta,\F)}{J}\sum_{j=1}^{J} \norm{\prs{S_{\lambda},\T_{\btheta_{j}^*} \bZ_{j}}}^2_{L^2}.
\end{align*}
Hence, conditionally  on $\btheta^*$ we have that
$
\sup_{\btheta\in\Theta^J}Q^{Z}_{\lambda}(\btheta) \leq  C(\Theta,\F)   \xi_{\lambda,J}\big( A_1,\ldots, A_J \big),
$
where $\xi_{\lambda,J}\big( A_1,\ldots, A_J \big)$ is defined in  Lemma \ref{lemme:conNoise} with  $A_j = \E_{\btheta^*}\big[\T_{\btheta_j^*}\bZ_j (\T_{\btheta_j^*}\bZ_j)'\big]^{\frac{1}{2}}$.  Let us now give an upper bound on the largest eigenvalues of the matrices $\mathbf A_j = A_j \mathbf{S}_\lambda A_j$ with $\mathbf S_{\lambda} = \big[\langle S^{\ell}_{\lambda}, S^{\ell'}_{\lambda}\rangle_{L^2}\big]_{\ell,\ell'=1}^n$.  Under Assumption \ref{hyp:EigZmax} we have that $\tr(\mathbf A_j) \leq \gamma_{\max}(A_j) \tr{\mathbf S_\lambda} \leq \gamma_n(\Theta) V(\lambda)$, for all $j=1,\ldots,J$ and any $\btheta^* \in  \Theta^J$. Thus, the result follows  by arguing as in the proof of Lemma \ref{lemme:conNoise} and by taking expectation with respect to $\btheta^*$. 
\end{proof}

\end{appendix}


\bibliographystyle{alpha}
\bibliography{FrechetLower_biblio}

\end{document}